\documentclass[12pt]{article}

\usepackage{amsthm,amsmath,stmaryrd,bbm,hyperref,geometry,color}
\usepackage{amssymb}
\usepackage[english]{babel}
 \usepackage[utf8]{inputenc}
\usepackage{graphicx}
\usepackage{amsfonts,amssymb}
\usepackage{verbatim}
\usepackage{enumitem}

\newcommand{\po}{\left(}
\newcommand{\pf}{\right)}
\newcommand{\co}{\left[}
\newcommand{\cf}{\right]}
\newcommand{\cco}{\llbracket}
\newcommand{\ccf}{\rrbracket}
\newcommand{\R}{\mathbb R} 
\newcommand{\T}{\mathbb T} 
\newcommand{\Z}{\mathbb Z} 
\newcommand{\A}{\mathcal A} 
\newcommand{\N}{\mathbb N} 
\newcommand{\dd}{\text{d}}
\newcommand{\na}{\nabla}
\newcommand{\1}{\mathbbm{1}}
\newcommand{\bx}{\mathbf{x}}
\newcommand{\by}{\mathbf{y}}
\newcommand{\bX}{\mathbf{X}}
\newcommand{\bY}{\mathbf{Y}}
\newcommand{\nv}[1]{#1}

\newtheorem{thm}{Theorem}

\newtheorem{prop}[thm]{Proposition}

\newtheorem{cor}[thm]{Corollary}

\newtheorem{assu}{Assumption}

\title{Wasserstein contraction and Poincaré inequalities for elliptic diffusions \nv{with high diffusivity}}
\author{Pierre Monmarché}

\begin{document}

\maketitle

\begin{abstract}
We consider elliptic diffusion processes on $\mathbb R^d$. Assuming that the drift contracts distances outside a compact set, we prove that, \nv{when the diffusion coefficient is sufficiently large}, the Markov semi-group associated to the process is a contraction of the $\mathcal W_2$ Wasserstein distance, which implies a Poincaré inequality for its invariant measure. The result doesn't require neither reversibility nor an explicit expression of the invariant measure, and the estimates have a sharp dependency on the dimension. Some variations of the arguments are then used to study, first, the stability of the invariant measure of the process with respect to its drift  and, second,  systems of interacting particles, yielding a criterion for dimension-free Poincaré inequalities and quantitative long-time convergence for non-linear McKean-Vlasov type processes.
\end{abstract}


\medskip

\noindent\textbf{MSC class:} 60J60

\section{Overview}\label{sec1}

Consider $(X_t)_{t\geqslant0}$ a diffusion process on $\R^d$ solution to
\begin{equation}\label{eq:EDS}
\dd X_t  = b(X_t) \dd t + \sqrt{2T}\dd B_t
\end{equation}
where $b\in\mathcal C^1(\R^d)$, $T>0$ is a constant and $(B_t)_{t\geqslant0}$ is a standard Brownian motion. Denote by $P_t$ the associated semi-group, namely $P_t f(x) = \mathbb E_x\po f(X_t)\pf$ for all suitable $f$ on $\R^d$. Let
\begin{equation}\label{eq:defk}
k(x):=  - \sup \left\{ \frac{ (x-y)\cdot \po b(x)-b(y)\pf }{|x-y|^2},\ y\in\R^d,\ y\neq x\right\}\,.
\end{equation}
We are interested in cases where the following holds:
\begin{assu}\label{Assu1}
There exist $K,R\geqslant 0$ and $c>0$ such that 
\begin{equation}\label{eq:conditionk}
k(x) \geqslant -K\quad \forall x\in\R^d \qquad \text{and}\qquad k(x) \geqslant c \quad \forall x\in\R^d\text{ with }|x|\geqslant R\,.
\end{equation}
\end{assu}
 Under this condition, it is standard to check that the process is non-explosive, admits a unique invariant measure $\mu$ with a positive Lebesgue density and that the law of the process converges to $\mu$ as $t\rightarrow \infty$, using e.g. Lyapunov/Doeblin conditions \cite{HairerMattingly2008,MeynTweedieLivre,BakryCattiauxGuillin}. The first main problem considered in this work is to prove that $\mu$ satisfies a Poincar\'e inequality, namely that  there exists a constant $C_P>0$ such that, for all $f\in\mathcal C^1(\R^d)$, writing $\mu f= \int_{\R^d} f \dd \mu$,
\[\| f - \mu f\|_{L^2(\mu)}^2 \leqslant C_P \| \na f\|_{L^2(\mu)}^2\,.\]
This inequality is related to concentration inequalities for $\mu$ and to the long-time convergence in $L^2(\mu)$ of the law of the process toward $\mu$ (see e.g. \cite{BakryGentilLedoux,CattiauxGuillinDeviation} and  below). More precisely, for $\mu \in \mathcal P(\R^d)$ we write
\[C_P(\mu) := \po \inf\{ \| \na f\|_{L^2(\mu)}^2\,,  \ f\in  \mathcal C^1(\R^d),\ \| f - \mu f\|_{L^2(\mu)} =1\}\pf^{-1} \]
the optimal constant in the inequality.

When $b=-\na U$ for some $U\in\mathcal C^2(\R^d)$, Assumption~\ref{Assu1} is equivalent to say that $U$ is convex outside a compact set, and then a Poincaré inequality is known to hold. In fact, in this case, $\mu$ has an explicit density, proportional to $\exp(-U/T)$, and moreover the process is reversible, namely its generator
\begin{equation}\label{eq:generateur}
L  = T\Delta + b\cdot \na
\end{equation}
is self-adjoint in $L^2(\mu)$. The Poincaré  constant is then exactly the spectral gap of  $L$ (in non-reversible cases, the spectral gap may be larger).  Many tools are available to establish Poincaré inequalities for reversible diffusions. In particular, under Assumption~\ref{Assu1}, a Poincaré inequality can be obtained by combining the Bakry-Emery curvature criterion and the Holley-Stroock perturbation argument (see e.g. \cite[Propositions 4.2.7 and Proposition 4.8.1]{BakryGentilLedoux} or Section~\ref{Sec:reversible} below) or a local inequality/Lyapunov condition as in  \cite{CattiauxGuillin,BakryCattiauxGuillin}, see also  \cite{BakryGentilLedoux,CattiauxGuillinFathi,BakryCattiauxGuillin} and references within concerning the reversible case.

Notice that different drifts $b$ can give the same invariant measure, for instance if $b = -(I_d +J)\na U$ where $J$ is any skew-symmetric matrix then $\exp(-U/T)$ is invariant for the process. Assumption~\ref{Assu1} depends on $b$, but the Poincaré inequality depends only on $\mu$.

 Now, if $b$ is not a gradient and if $\mu$ is not explicit, much less is known. If $k(x)\geqslant c>0$ for all $x\in\R^d$ then the Bakry-Émery arguments  still works \cite{MonmarcheContraction,CattiauxGuillinSLC}. If $k(x) \geqslant c>0$ only for $|x|$ large enough, a Poincaré inequality should be expected, but to the best of our knowledge it cannot be established by existing methods.

This issue leads to the second main question of this work, which is to prove that $P_t$ is a contraction of the $\mathcal W_2$ Wasserstein distance for $t$ large enough. Indeed, from classical arguments   (see Section~\ref{sec:preuveL2}), this  implies a Poincaré inequality for $\mu$. Recall that, for $\alpha\in[1,\infty)$, the $\mathcal W_\alpha$ Wasserstein distance between  $\nu,\nu'\in\mathcal P(\R^d)$ (the set of probability measures on $\R^d$) is defined by
\[\mathcal W_\alpha(\nu,\nu') = \inf_{\pi\in\Pi(\nu,\nu')} \po \int_{\R^d} |x-y|^\alpha \pi(\dd x,\dd y)\pf^{1/\alpha}\,,\]
where $\Pi(\nu,\nu')$ is the set of probability measures on $\R^d\times\R^d$ with marginals $\nu$ and $\nu'$. Writing $\nu P_t$ the law at time $t$ of a process solving \eqref{eq:EDS} with an initial condition $X_0$ distributed according to $\nu$ (so that $(\nu P_t) f = \nu(P_t f)$ for all bounded measurable $f$), we want to find $M,\lambda>0$ such that
\begin{equation}\label{eq:contractW}
\forall t\geqslant 0,\ \forall \nu,\nu'\in\mathcal P(\R^d)\,,\qquad 
\mathcal W_\alpha(\nu P_t,\nu' P_t) \leqslant M e^{-\lambda t} \mathcal W_\alpha(\nu,\nu'),
\end{equation}
in particular for $\alpha=2$. Such a contraction of $\mathcal W_\alpha$ implies a contraction of $\mathcal W_{\beta}$ for $\beta\leqslant \alpha$, see e.g.  \cite{MonmarcheContraction}. Under Assumption~\ref{Assu1}, such a contraction can be proven for $\alpha=1$  using a reflection coupling and a concave modification of the distance, see \cite{Eberle1,Eberle2}.  However, due to the convexity at $0$ of $r\mapsto r^2$, for $\alpha=2$ the method only yields   estimates of the form $\mathcal W_2(\nu P_t,\nu' P_t) \leqslant M e^{-\lambda t} \max( \mathcal W_{2}(\nu,\nu'), \mathcal W_1^{1/2}(\nu,\nu'))$ (see \cite{JWang2,JWang1}), which are weaker than \eqref{eq:contractW} with $\alpha=2$ and are not sufficient to get a Poincaré inequality in the non-reversible case.

  On the other hand, when $k(x) \geqslant c >0$ for all $x\in\R^d$, the contraction \eqref{eq:contractW} holds with $M=1$ and $\lambda=c$ for all $\alpha$, see \cite{MonmarcheContraction}. In fact this is an equivalence\nv{, according to the Sturm-Von Renesse Theorem \cite{VonRenesseSturm} (which is also true in the present non-reversible case \cite{MonmarcheContraction}):}  if there exists $\alpha\geqslant 1$ such that \eqref{eq:contractW} holds with $M=1$ and some $\lambda\in\R$,  then necessarily it holds for all $\alpha$ and it implies that $k(x) \geqslant \lambda$ for all $x\in\R^d$. \nv{Moreover, it implies a so-called log-Sobolev inequality for the invariant measure, which is stronger than the Poincaré one (see \cite{MonmarcheContraction} for details)}. \nv{Since,} again, we are interested in the case where $k(x)\geqslant c>0$ only holds for $x$ large enough, it means in particular that in that case \eqref{eq:contractW} with a positive $\lambda$ can only hold with $M>1$.

\medskip

Our main result is the following:
\begin{thm}\label{Thm:main}
Under Assumption~\ref{Assu1}, suppose furthermore that
\begin{equation}\label{eq:conditionT}
T\geqslant  T_0:=(2K+c)\frac{ \alpha (K+c/4)R_*^2   +2  \sup\left\{-x\cdot b(x),\ |x|\leqslant R_*\right\}}{cd}  
\end{equation}
for some $\alpha \geqslant 2$, where $R_*=R\po 2+2K/c\pf^{1/d}$. Then  \eqref{eq:contractW} holds with
\[\lambda  = \frac c4 \,, \qquad M= M_\alpha := \po 1+ \frac{\alpha (2K+c)R_*^2}{4dT}\pf^{1/\alpha}\,.\]
\end{thm}

This is proven in Section~\ref{Sec:preuveThmMain}. In contrast to the Poincaré inequality, this result is new even in the reversible case.


\nv{Let us now state some implications of a $\mathcal W_2$ contraction when $M>1$,} obtained from known arguments (see  Section~\ref{Sec:AutresPreuves} for the proof of the next result). To avoid technical discussions, we assume that the coordinates of the force fields $b$ are in $\mathcal A$ the set of $\mathcal C^\infty$ functions from $\R^d$ to $\R$ with all derivatives growing at most polynomially at infinity (with a slight abuse of notation we simply write $b\in\A$ in that case), and we only consider test functions in $\mathcal A$. Combined with Assumption~\ref{Assu1} which implies a time-uniform Gaussian moment for $X_t$ via standard Lyapunov arguments, we get that, for all $f\in\A$ and all $t\geqslant 0$, $L\in\A$, $P_t f\in \A$ and $\partial_t P_t f = LP_t f = P_t Lf$ (see e.g. the proof of \cite[Theorem 2.5]{EthierKurtz}).

\begin{thm}\label{Thm:main2}
\nv{Assume that $k(x) \geqslant - K$ for all $x\in\R^d$ for some $K\geqslant 0$,}  that $b\in\A$ and that a Wasserstein contraction \eqref{eq:contractW} holds with $\alpha=2$ for some $M\geqslant1 ,\lambda>0$. Then:
\begin{enumerate}
\item The invariant measure $\mu$ satisfies a Poincaré inequality with
\[C_P(\mu) \leqslant \frac{M^2 T}{\lambda}\,.\]
\item For all $f\in \A$ and all $t\geqslant 0$,
\begin{equation}\label{eq:Variancedecay}
\|P_t f -\mu f\|_{L^2(\mu)} \leqslant \min \po e^{-tT/C_P(\mu)}, \po 1+ \nv{\frac{T(e^{2\lambda t}-1)}{C_P(\mu)M^2\lambda }}\pf^{\nv{-1/2}}\pf \|f -\mu f\|_{L^2(\mu)}\,.
\end{equation}
\item For  all $t>0$ and \nv{any} probability law $\nu$ on $\R^d$ with finite second moment, $\nu P_t$ has a density $h_t$ with respect to $\mu$ and 
\begin{equation}\label{eq:thm2Wang}
\|\nu P_t - \mu\|_{TV}^2 \ \leqslant\ \mu (h_t\ln h_t) \  \leqslant \ \frac{J(t)}{2T} \mathcal W_2^2(\nu,\mu)\,,
\end{equation}
where $J(t) \leqslant K/(1-e^{-2Kt})$ for all $t> 0$ and
\[J(t) \leqslant  M^2 (K+\lambda) \po 1+K/\lambda\pf^{\lambda/K}e^{-2\lambda t}\qquad \text{for all } t \geqslant \frac{\ln(1+K/\lambda )}{2K}\]
if $K>0$, and the limit of these expressions as $K\rightarrow 0$ if $K=0$.
\item If furthermore $b=-\na U$ for some $U\in\mathcal C^2(\R^d)$ then  $C_P(\mu)\leqslant T/\lambda$. 
\end{enumerate}
\end{thm}

\nv{In both Theorems\ref{Thm:main} and \ref{Thm:main2}, keep in mind that $\mu$ and $P_t$ depends on $T$.}

As we see, we only have a partial answer to our initial questions, since the results only hold for $T$ large enough (or, equivalently, for $R$ small enough, which means the results hold for small perturbations of the case where $k(x)\geqslant c$ for all $x\in\R^d$). We didn't try to make the condition \eqref{eq:conditionT} on $T$ as sharp as possible: it can be slightly improved, but it cannot be suppressed simply by optimizing our proof.  We do not know whether, for $T=1$, Assumption~\ref{Assu1} is sufficient to get \eqref{eq:contractW}  for some $M,\lambda$ or to get  a Poincaré inequality for $\mu$. However, notice that, for $T\geqslant T_0$, Theorem~\ref{Thm:main} gives another important information, which is that the contraction~\eqref{eq:contractW} and the Poincaré inequality hold respectively with   $M,\lambda$ and $C_P(\mu)/T$ which are uniform in $T\geqslant T_0$. Now, this part is clearly false if we suppress the condition that $T$ has to be large enough, namely the statement \emph{``Under Assumption~\ref{Assu1}, there exist $M,\lambda>0$ such that, for all $T>0$, \eqref{eq:contractW} holds"} is clearly false, and so is \emph{``Under Assumption~\ref{Assu1}, there exists $C>0$ such that, for all $T>0$, $C_P(\mu) \leqslant
 TC$"}. Indeed, the first statement would imply the second (according to Theorem~\ref{Thm:main2}), and  it is well known that,  if for instance $b=-\na U$ where $U$ has several isolated local minima, then $C_P(\mu)\geqslant e^{a/T}$ for some $a>0$ for $T$ small enough \cite{HKS,MenzSchlichting}.

Interestingly, apart from the dependency on $T$, the bounds on $M,\lambda$ and $ C_P(\mu)$ given by Theorems~\ref{Thm:main} and \ref{Thm:main2} behave rather well with the dimension $d$, in contrast to what usually give the methods based on the existence of a Lyapunov function and of a local Poincaré inequality \cite{CattiauxGuillin,BakryCattiauxGuillin} or on the perturbation of a reference measure \cite{HKS}. For instance, consider the probability measure $\mu\propto \exp(-U)$ with $U(x)=|x|^\beta/\beta$, for $\beta>2$.  By applying Theorems~\ref{Thm:main} and \ref{Thm:main2}, we get that 
\begin{equation}\label{eq:Cpbeta}
C_P(\mu) \leqslant  \frac{ 8  (\beta-1) \po 1+ 2^{\beta/d+1} \pf^{1-2/\beta}}{d^{1-2/\beta}}\,,
\end{equation}
see Section~\ref{Sec:convexDegenere}. By contrast, using a standard curvature$+$bounded perturbation argument, one cannot get better than what is given by the curvature result, which is dimension-free (see  Section~\ref{Sec:reversible} or \cite[Remark 5.21]{CattiauxGuillinSLC}). Besides, here, $\mu$ is a radial log-concave probability measure, for which two-sided bounds on the Poincaré constants are known \cite{Bobkov,BonnefontJoulinMa}, in relation to the KLS conjecture \cite{ChenKLS,LeeVempala}. In particular, for $U(x) =|x|^\beta/\beta$, \cite[Corollary 4.2]{BonnefontJoulinMa} reads
\[\frac{d}{(d+2) d^{1-2/\beta}} \leqslant  C_P(\mu) \leqslant \frac{d+1}{(d-1) d^{1-2/\beta}}  \]
for $d\geqslant 2$. Hence, the dependency in the dimension $d$ in \eqref{eq:Cpbeta} (which is based on our general result and thus does not use that $\mu$ is radial) is  optimal.

\nv{As a last remark on Theorem~\ref{Thm:main2}, notice that, in \eqref{eq:Variancedecay}, the minimum is always given by $e^{-tT/C_P(\mu)}$ in the reversible case. However, there are non-reversible cases where $\lambda> T/C_P(\mu)$, so that the second term becomes smaller for large $t$. For instance, in the Gaussian case $\mu \propto \exp(- U/T)$  with $U(x) = x\cdot Ax$ for some definite positive symmetric  matrix $A$, denote by $\nu_1,\dots,\nu_d$ the eigenvalues of $A$. Then it is well known that $T/C_P(\mu) = \min\{ \nu_i,i\in\cco 1,d\ccf\}$, while non-reversible Gaussian processes with invariant measure $\mu$ are constructed in \cite{Lelievre2012} with  a linear drift $-Bx$ where the real parts of the eigenvalues of the matrix $B$ are larger than $\bar\nu := (\nu_1+\dots+\nu_d)/d$,  so that a $\mathcal W_2$ contraction holds with $\lambda = \bar\nu$  (as can be seen using a synchronous coupling, see e.g. \cite{MonmarcheContraction}).   }

\medskip

Contrary to a simple long-time convergence at equilibrium in $\mathcal W_\alpha$, a contraction  of $\mathcal W_\alpha$ can easily lead to perturbative results.  For instance, consider on $\R^d$ a  continuous process solving
\begin{equation}\label{eq:EDSpertubée}
\dd Y_t  = \tilde b(Y_t,Z_t) \dd t + \sqrt{2T}\dd B_t
\end{equation}
where $Z=(Z_t)_{t\geqslant 0}$ is a random càdlàg process on some state space $E$ and $\tilde b :\R^d\times E\rightarrow \R^d$. Denote by $\tilde\nu_t$ the law of $Y_t$. 

\begin{prop}\label{prop:perturbation}
Let $b$ be a $\mathcal C^1$ vector field on $\R^d$ satisfying Assumption~\ref{Assu1} and $P_t$ be its associated semi-group. Let $\alpha\geqslant 2$. Assume that $T\geqslant T_0$ with $T_0$ given by \eqref{eq:conditionT}. Then, for all $\nu\in\mathcal P(\R^d)$ and $t\geqslant 0$,
\begin{equation}\label{eq:perturb}
\mathcal W_\alpha(\nu P_t,\tilde \nu_t) \leqslant M_\alpha e^{-\lambda t} \mathcal W_\alpha(\nu ,\tilde \nu_0)+ M_\alpha^\alpha  \int_0^t  e^{\lambda (s-t)}  \po \mathbb E \po | b(Y_s) - \tilde b(Y_s,Z_s)|^\alpha\pf\pf^{1/\alpha}\dd s\,.
\end{equation}
where $\lambda,M_\alpha$ are as in Theorem~\ref{Thm:main}.
\end{prop}

This is proven in Section~\ref{sec:perturbation}. The right hand side in \eqref{eq:perturb} can be bounded given additional informations on $\tilde b$, for instance if $\tilde b(y,z)= \tilde b(y)$ and we simply assume that $\|b-\tilde b\|_\infty < \infty$ as in \cite{EberleZimmer,DiscreteSticky} then
\[\mathcal W_\alpha(\nu P_t,\tilde \nu_0 \widetilde{P_t}) \leqslant M_\alpha e^{-\lambda t} \mathcal W_\alpha(\nu ,\tilde \nu_0)+ M_\alpha^\alpha  \frac{1- e^{\nv{-\lambda t}}}\lambda \|b-\tilde b\|_\infty\,,\]
with $\widetilde{P_t}$ the semi-group associated to $\tilde b\cdot\na + T\Delta$. In particular, any invariant measure $\tilde\mu$ of $\widetilde{P_t}$ satisfies
\[\mathcal W_\alpha(\mu,\tilde \mu ) \leqslant    \frac{M_\alpha^\alpha}\lambda \|b-\tilde b\|_\infty\,.\]
Hence, under our restrictive condition \eqref{eq:conditionT}, we extend the results of \cite{EberleZimmer}, which are restricted to $\alpha=1$.

More generally, the right hand side in \eqref{eq:perturb} can be bounded under the assumption that $\mathbb E(|b(y)-\tilde b(y,Z_t)|^\alpha) \leqslant Q(|y|)$ for all $y\in\R^d$ for some polynomial $Q$ and then with some moment estimates on $Y_t$ obtained by Lyapunov arguments (see e.g. the proof of Theorem~\ref{thm:champmoyen} below). 

A case of particular interest, in view both of the perturbation result and of the condition on the \nv{diffusion coefficient $T$ (to be thought as a temperature parameter in statistical physics)}, is given by systems of interacting particles, detailed in Section~\ref{sec:particules}.

\medskip

As a summary, the rest of this paper is organized as follows. Section~\ref{Sec:preuveThmMain} is devoted to the proof of Theorem~\ref{Thm:main}. Section~\ref{Sec:AutresPreuves} gathers the proofs of the other results stated in this introduction, namely Theorem~\ref{Thm:main2}, the Poincaré inequality \eqref{eq:Cpbeta} and Proposition~\ref{prop:perturbation}, and a discussion on the reversible case. System of interacting particles  are studied  in Section~\ref{sec:particules}. Finally, we conclude this work in Section~\ref{sec:discussion} by an informal discussion on our method, related works and possible perspectives.

\section{Proof of the main theorem}\label{Sec:preuveThmMain}

 Assumption~\ref{Assu1} is enforced in all this section, devoted to the proof of Theorem~\ref{Thm:main}.
 
 \subsection{A probabilistic proof}\label{sec:demo-couplage}

\subsubsection{Synchronous coupling and modified \nv{cost}}\label{Sec:preuve_couplage}
For $(B_t)_{t\geqslant 0}$ a standard Brownian motion on $\R^d$, we consider $(X_t,Y_t)_{t\geqslant 0}$ the Markov process on $\R^d\times\R^d$ solution to
\[\dd X_t  = b(X_t) \dd t + \sqrt{2T}\dd B_t\,\qquad \dd Y_t  = b(Y_t) \dd t + \sqrt{2T}\dd B_t\,,\]
which is called the parallel or synchronous coupling of two diffusions \eqref{eq:EDS}. The generator $\mathcal L_s$ of this process is given by 
\[\mathcal L_s f(x,y) = b(x)\cdot\na_x + b(y)\cdot \na_y + T \na \cdot \po A \na f\pf (x,y) \qquad \text{where} \qquad A= \begin{pmatrix} I_d & I_d \\ I_d & I_d\end{pmatrix} \,.\]
Given $\alpha\geqslant 2$ and a bounded positive $\nv{\omega}\in\mathcal C^2(\R^d)$, consider
\[\rho(x,y)=|x-y|^\alpha \po 2T + \alpha \nv{\omega}(x)+ \alpha \nv{\omega}(y)\pf\,,\]
\nv{which is a modification of the usual transport cost $|x-y|^\alpha$ for the Wasserstein distances.} 
Using that $A \na f(x,y)=0$ for $f(x,y)=|x-y|^2$, we get
\begin{eqnarray*}
\mathcal L_s \rho(x,y) & =&  \alpha |x-y|^{\alpha-2} (x-y)\cdot \po b(x)-b(y)\pf \po 2T + \alpha \nv{\omega}(x)+ \alpha \nv{\omega}(y)\pf + \alpha |x-y|^{\alpha} \po L\nv{\omega}(x)+L\nv{\omega}(y)\pf\\
& \leqslant &   \alpha |x-y|^{\alpha} \po \Psi(x)+\Psi(y)\pf\,,
\end{eqnarray*}
with
\begin{equation}\label{eq:Psi}
\Psi(x) = - k(x) \po T + \alpha \nv{\omega}(x)\pf +  L\nv{\omega}(x)\,.
\end{equation}
For now, assume that $\nv{\omega}$ is such that there exists $\lambda>0$ such that, 
\begin{equation}\label{eq:conditionPsi}
\forall x\in\R^d\,,\qquad \Psi(x) \leqslant -\lambda    \po T+\alpha \nv{\omega}(x) \pf\,.
\end{equation}
 Then, $\mathcal L_s \rho + \alpha \lambda \rho \leqslant 0$, i.e. $(e^{\alpha \lambda t}\rho(X_t,Y_t))_{t\geqslant 0}$ is a submartingale and for all $t\geqslant 0$,
 \[\mathbb E\po \rho \po X_t,Y_t\pf\pf \leqslant e^{-\alpha\lambda t}\mathbb E \po \rho(X_0,Y_0)\pf\,.\]
Set
 \[M= \po 1+ \frac{\alpha \|\nv{\omega}\|_\infty}{T}\pf^{1/\alpha} \,.\]
 Let $\nu,\nu' \in \mathcal P(\R^d)$. For any $\pi_0\in\Pi(\nu,\nu')$, considering an initial condition $(X_0,Y_0)\sim \pi_0$ independent from $(B_t)_{t\geqslant 0}$, we obtain
 \[\mathcal W_\alpha\po \nu P_t,\nu' P_t \pf \leqslant \po\mathbb E\po |X_t-Y_t|^\alpha \pf\pf^{1/\alpha} \leqslant M e^{-\lambda t} \po\mathbb E_{\pi_0}\po |X_0-Y_0|^\alpha \pf\pf^{1/\alpha}\,.\]
Finally, taking the infimum over $\pi_0\in\Pi(\nu,\nu')$ yields  \eqref{eq:contractW}. The proof is thus complete if we are able to construct a bounded positive $\nv{\omega} \in \mathcal C^2(\R^d)$ such that \eqref{eq:conditionPsi} holds for some $\lambda>0$. This is the content of the next section.

 \subsubsection{The weight function $\nv{\omega}$}\label{sec:simplecondition}

\begin{figure}
\begin{minipage}{0.45\linewidth}
\includegraphics[scale=0.3]{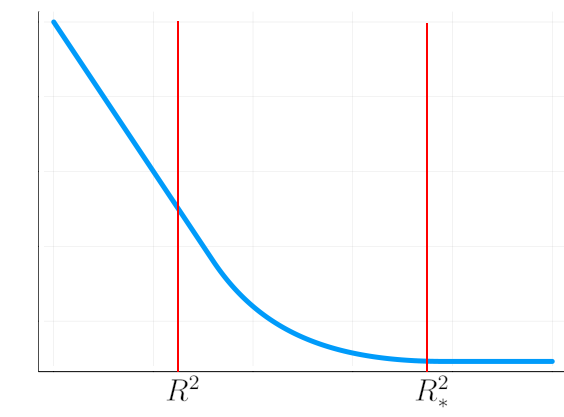}
\end{minipage}
\begin{minipage}{0.45\linewidth}
\includegraphics[scale=0.3]{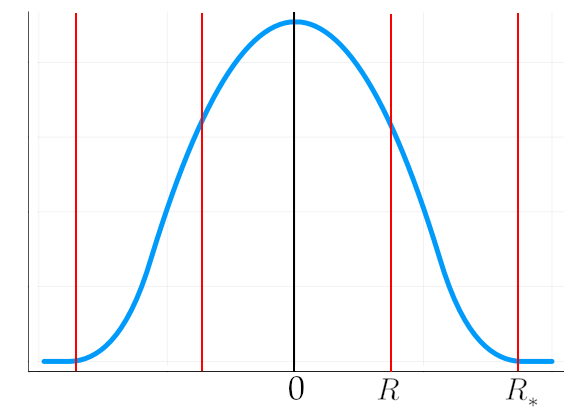}
\end{minipage}
\caption{Left: $r\mapsto g(r)$. Right: $x \rightarrow \omega(x) = g(|x|^2) - \min g$ (in dimension 1). For $r\leqslant R^2$, $g$ is affine decreasing, for $r\geqslant R_*^2$, it is constant, and in between it is convex but with $g''$ constrained not to be too large, which thus requires to take $R_*$ large enough. }\label{figure}
\end{figure}

Our goal is to construct a bounded positive function $\nv{\omega}$ such that
\[\forall x\in\R^d, \quad \Delta \nv{\omega} (x)  \leqslant    k(x) -\frac{c}{2} \qquad\text{and}\qquad Q:= \sup_{x\in\R^d} b(x)\cdot \na \nv{\omega}(x) <\infty \,. \]
Indeed, if this holds, taking $\lambda=c/4$, we get 
\begin{eqnarray*}
\Psi(x) + \lambda    \po T+\alpha \nv{\omega}(x) \pf  & = &  \po \lambda - k(x)\pf \po T + \alpha \nv{\omega}(x)\pf +  L\nv{\omega}(x)\\
& \leqslant & -\frac{c}{4} T + \alpha (K+\lambda) \|\nv{\omega}\|_\infty +Q
\end{eqnarray*}
and thus \eqref{eq:conditionPsi} holds if $T\geqslant 4\po \alpha (K+\lambda) \|\nv{\omega}\|_\infty +Q\pf/c$.

We take $\nv{\omega}$ of the form $\nv{\omega}(x) = g(|x|^2)-\inf g$ with $g\in\mathcal C^2(\R_+)$ to be chosen. Since
\[\na\nv{\omega}(x) = 2xg'(|x|^2) \,,\qquad  \Delta \nv{\omega} = 4|x|^2 g''(|x|^2) +  2dg'(|x|^2)\,,\]
setting $K_* = K/4+c/8$, we  take $g$ as the $\mathcal C^1$ solution of $g(0)=0$, $g'(0)=-2K_*/d$ and
\[r g''(r) +  \frac{d}2 g'(r) =  h(r):= \left\{\begin{array}{ll}
-K_*  & \text{for }r<R^2\\
c/8 & \text{for }r\in(R^2,R_*^2)\\
0 & \text{for }r>R_*^2
\end{array}\right.\]
where  $R_*>0$ remains to be chosen  so that $g'(R_*^2)=0$ (and thus $g'(r)=g''(r)=0$ for all $r> R_*^2$). \nv{See Figure~\ref{figure} for a draft of the graph of   $g$ and $\omega$.} Notice that $g$ is not $\mathcal C^2$, but this is easily solved by replacing this $g$ by some $g_\varepsilon \in\mathcal C^2(\R_+)$ such that, for a small $\varepsilon>0$, first, $g''_\varepsilon(r)=g'_\varepsilon(r)=0$ for $r\geqslant R_*^2 + \varepsilon$ and, second, 
$r g_\varepsilon''(r) +  \nv{(d/2)} g_\varepsilon'(r)$  is  always smaller than $h(r)$, equal to $ -K_*$ for $r\in[0,R^2]$ and to $c/8$ for $r\in[R^2+\varepsilon,R_*^2-\varepsilon]$. We can conclude the proof with $g_\varepsilon$ and finally let $\varepsilon$ vanish in the final result. For simplicity we write the proof directly with $g$.

For $r\in[0,R^2]$, we simply have $g'(r) = -2K_*/d$.  For $r\geqslant R^2$, using that  $\po r^{d/2} g'(r)\pf' = h(r) r^{d/2-1}$, we get
\[r^{d/2} g'(r)  = R^{d} g'(R^2) +  \frac c8\int_{R^2}^r s^{d/2-1}\dd s = -\frac{2K_*}{d}R^{d} +  \frac{c}{4d}\po r^{d/2} - R^{d}\pf \,,\]
and thus we \nv{choose}
\[    R_*  = R \po 1+ 8K_*/c\pf^{1/d} = R\po 2+2K/c\pf^{1/d}\,, \]
which concludes the construction of $\nv{\omega}$. It remains to estimate $Q$ and $\|\nv{\omega}\|_\infty$. Since $g$ is decreasing,  constant on $[R_*^2,+\infty)$ and such that $\|g'\|_\infty =2K_*/d$, we get
\[\forall r\geqslant0\,,\qquad  0 \geqslant g(r) \geqslant  g(R_*^2) \geqslant -\frac{2K_*R_*^2}{d}\,, \]
hence $\|\nv{\omega}\|_\infty \leqslant 2K_* R_*^2/d$, and
\[ Q  \leqslant  \frac{4K_*}{d} \sup\left\{-x\cdot b(x),\ |x|\leqslant R_*\right\}\,.\]
The proof is concluded by using these estimates in the definition of $M$ and the condition on $T$.

\subsection{Alternative proof with Bakry-Emery interpolation}\label{Sec:preuveAlternative}

In this section we give a second proof of Theorem~\ref{Thm:main} (see Section~\ref{Sec:reflexionProof} for a discussion on the specific interest of each proof). \nv{This proof is similar to the intertwinning method of Joulin, Bonnefont and their coauthors \cite{arnaudon2018intertwinings,bonnefont2022note,BonnefontJoulinMa,10.3150/12-BEJ433} (who focus on reversible cases). } For simplicity we only consider the case $\alpha=2$, which is the main case of interest due to Theorem~\ref{Thm:main2}. The expressions of $T_0$ and $M$ obtained along this alternative proof are slightly different than those stated in Theorem~\ref{Thm:main} and established in the first proof (again, we don't try to optimize the estimate on $T_0$). Moreover, in order to justify the time \nv{derivatives} in this section, we assume that the coordinate \nv{functions} of $b$ are in $\A$.

The proof is organized in three steps: first we rephrase Assumption~\ref{Assu1} as a local condition on the drift, namely a condition on its Jacobian. Second, similarly to Section~\ref{Sec:preuve_couplage}, we give the proof conditionally to the existence of a suitable weight function. Third, similarly to  Section~\ref{sec:simplecondition}, we construct the weight function.

\subsubsection{Step 1: \nv{an infinitesimal condition}} \label{sec:localCondition}

We start by an equivalent formulation of Assumption~\ref{Assu1}. We denote by $\nv{D b}$ the Jacobian matrix of $b$ and
\[\tilde k(x) = -\sup\{ u\cdot \nv{D b}(x) u,\ u\in\R^d,\ |u|=1\}\,.\]
It is clear that, for all $x\in\R^d$,
\begin{equation}\label{eq:ktilde}
\tilde k(x) \geqslant k(x) \geqslant \inf_{y\in\R^d} \int_{0}^1 \tilde k(x+ty) \dd t\,.
\end{equation}
In particular under Assumption~\ref{Assu1}, $\tilde k$ also satisfies \eqref{eq:conditionk}, with the same $K,R,c$. Alternatively, assume that there exist $K,R\geqslant 0$ and $c>0$ such that
\begin{equation}\label{eq:assu:alternative}
\tilde k(x) \geqslant -K\quad \forall x\in\R^d \qquad \text{and}\qquad  \tilde k(x) \geqslant c \quad \forall x\in\R^d\text{ with }|x|\geqslant R\,.
\end{equation}
Then, from \eqref{eq:ktilde}, $k(x) \geqslant - K$ for all $x\in\R^d$. Moreover, for all $x,y\in\R^d$,  the Lebesgue measure of $\{t\in[0,1], |x+ty|<R\}$ being less than $2R/(|x|+R)$,
\[k(x) \geqslant \inf_{y\in\R^d} \int_{0}^1 \tilde k(x+ty) \dd t \geqslant -K \frac{2R}{|x|+R} + c \po 1 - \frac{2R}{|x|+R}\pf \geqslant \frac{c}2\] 
for all $x\in \R^d$ with $|x|\geqslant \tilde R:= R\po 4(K+c)/c-1\pf$. In other words, Assumption~\ref{Assu1} is equivalent to assume that \nv{the infinitesimal condition}  \eqref{eq:assu:alternative} holds for some $K,R\geqslant 0$ and $c>0$. This latter condition is enforced for the rest of Section~\ref{Sec:preuveAlternative}.

\subsubsection{Step 2: Gamma calculus}

 Consider a positive $a\in\A$, to be chosen later on. Fix $t>0$ and $f\in\A$. The carré du champ of the generator $L$ is defined as $\Gamma(f,g) = \frac12 \po L(gf) - g Lf-fLg\pf$, with the notation $\Gamma(f)=\Gamma(f,f)$.  In the case of \eqref{eq:generateur}, this is simply $\Gamma(f,g)=T\na f\cdot\na g$. We write $[A,B]=AB-BA$.  For a given $i\in\cco 1,d\ccf$, write $g_s = \partial_{x_i} P_{t-s} f$ for $s\in[0,t]$. Then
\begin{eqnarray*}
\partial_s \po P_s\po a g_s^2\pf\pf   &= & P_{s} \po L(ag_s^2) - 2a g_s \partial_{x_i} LP_s f \pf  \\
&= & P_{s} \po L(a)g_s^2  + aL(g_s^2) + 2 \Gamma(a,g_s^2)- 2a g_s ([\partial_{x_i} , L] P_{t-s} f + Lg_s)\pf  \\
&= & P_{s} \po L(a)g_s^2  + 2a\Gamma(g_s) + 2 \Gamma(a,g_s^2)+ 2a g_s [L,\partial_{x_i}] P_{t-s} f\pf\,.  
\end{eqnarray*}
Using that
\[[L,\partial_{x_i}] = [b\cdot \na,\partial_{x_i}] = -(\partial_{x_i} b) \cdot \na  \]
and
\[2\Gamma(a,g_s^2) = 4g_s \Gamma(a,g_s) \geqslant -2\frac{\Gamma(a)}{ a} g_s^2 - 2 a \Gamma(g_s) \]
(where $\Gamma(a)/a$ is understood as $0$ if $\Gamma(a)=0$), summing over $i\in\cco 1,d\ccf$, we get
\begin{eqnarray*}
\partial_s \po P_s \po a |\na P_{t-s}f|^2\pf  \pf &\geqslant  & P_{s} \po  \po L(a) - 2\frac{\Gamma(a)}{ a}\pf |\na P_{t-s} f|^2   - 2a (\na P_{t-s} f) \cdot \nv{D b} \na P_{t-s} f\pf  \\
 &\geqslant  & 2 P_{s} \po \Phi(a) |\na P_{t-s} f|^2\pf
\end{eqnarray*}
with
\[\Phi(a) =   \frac12 L(a) - \frac{\Gamma(a)}{ a} + a \tilde k\]
Assume for now that $a$ is such  that there exists $\lambda>0$ such that
\begin{equation}\label{eq:conditionaBakryEmery}
a \text{ and }a^{-1}\text{ are bounded and } \Phi(a) \geqslant  \lambda a \,.
\end{equation} 
\nv{(This condition is similar to the ones in \cite{arnaudon2018intertwinings,bonnefont2022note,BonnefontJoulinMa,10.3150/12-BEJ433}, for instance \cite[Theorem 3.2]{arnaudon2018intertwinings}; in this work, the authors do not assume that $a$ is bounded, as they are interested in Brascamp-Lieb inequalities, namely weighted Poincaré inequalities, but in the present work we are interested in classical Poincaré inequalities and contraction of the Wasserstein distances associated to the standard Euclidean distance, which is why we add this condition so that $a|\na f|^2$ is equivalent to $|\na f|^2$).} Integrating the previous inequality yields
\[P_t\po a|\na f|^2\pf   \geqslant e^{2\lambda t}  a |\na P_t f|^2\,, \] 
and then
\[|\na P_t f|^2 \leqslant \|a\|_\infty \| a^{-1}\|_\infty e^{-2\lambda t} P_t|\na f|^2\,,\]
which concludes the proof since,  thanks to the work of Kuwada \cite{Kuwada1,Kuwada2}, it is equivalent to the $\mathcal W_2$ contraction \eqref{eq:contractW} with the same $\lambda$, $\alpha=2$ and $M = \sqrt{\|a\|_\infty \| a^{-1}\|_\infty }$. \nv{Indeed, more precisely, in the case of a diffusion process on $\R^d$, as a corollary of \cite[Theorem 2.2]{Kuwada1} (applied with $v$ the Lebesgue measure), we get the following:
\begin{prop}\label{prop:kuwada}
Assume that $b\in \mathcal A$ and that $k(x) \geqslant - K$ for all $x\in\R^d$ for some $K\geqslant 0$. For $\rho\in\R$, $M>0$ and $t\geqslant 0$, the two following assertions are equivalent: 
\begin{itemize}
\item For all $f\in\mathcal A$,
\[|\na P_t f| \leqslant M e^{-\lambda t} \sqrt{  P_t|\na f|^2}\,.\]
\item For all probability measures  $\nu,\mu$ on $\R^d$,
\[\mathcal W_2(\nu P_t,\mu P_t) \leqslant M e^{-\lambda t} \mathcal W_2(\nu,\mu)\,.\]
\end{itemize}
\end{prop}
Notice that, in \cite[Theorem 2.2]{Kuwada1}, the equivalence is stated for the class of functions $f$ which are bounded and Lipschitz. For $f\in\mathcal A$, we can find a sequence $f_n$ of bounded Lipschitz functions $(f_n)_{n\in\N}$ such that $|\na (f_n - f)|(x) \leqslant \varepsilon_n + \1_{|x|\geqslant 1/\varepsilon_n} q(x)$ for some polynomial $q$, where $\varepsilon_n \rightarrow 0$ as $n\rightarrow 0$. Using a synchronous coupling and that the process admits Gaussian moments, it is easily seen that $|\na P_t (f-f_n)| \leqslant e^{Kt} P_t|\na (f-f_n)| \rightarrow 0 $, hence the result.
}

\subsubsection{Step 3: the weight function} 

 It remains to construct a weight $a$ satisfying \eqref{eq:conditionaBakryEmery} for some $\lambda>0$ under the condition~\eqref{eq:assu:alternative}.
 As in Section~ \ref{sec:simplecondition}, we focus on the leading term for large $T$, setting
\[a(x) = T + 2\po  \|\nv{\omega}\|_\infty - \nv{\omega}(x) \pf \,,\]
where $\nv{\omega}$ is a bounded positive function to be chosen. Then, using that $a\geqslant T$ \nv{and $\Gamma(\omega)=T|\na \omega|^2$,}
\[\Phi(a) \geqslant -T\Delta \nv{\omega} -  b\cdot \na \nv{\omega} - 4|\nabla \nv{\omega}|^2 + T \tilde k -2 K\|\nv{\omega}\|_\infty \,.\]
Let $\nv{\omega}$ be such that
\[\forall x\in\R^d, \quad \Delta \nv{\omega} (x)  \leqslant    \tilde  k(x) -\frac{c}{2} \qquad\text{and}\qquad \tilde Q:=  K \|\nv{\omega}\|_\infty +  \sup_{x\in\R^d}\po  b(x)\cdot \na \nv{\omega}(x) + 4|\na \nv{\omega}(x)|^2 \pf<\infty \,,\]
as constructed in Section~\ref{sec:simplecondition}. Then 
\[\Phi(a) \geqslant \frac {cT}2 - \tilde Q  \geqslant \frac{cT}{3} \geqslant \frac c4 a\]
if 
\[T\geqslant \tilde T_0 := 6 \max \po \frac{\tilde Q}{c},\|\nv{\omega}\|_\infty \pf \,.\]
Moreover, for $T\geqslant \tilde T_0$, we get $T\geqslant 3 a/4 \geqslant 3T/4$, hence  $\|a\|_\infty \|a^{-1}\|_\infty \leqslant 4/3$, i.e. \eqref{eq:contractW} holds with $M=\sqrt{4/3}$ and $\lambda=c/4$. An explicit upper bound of $\tilde T_0$ follows from the estimates on $\nv{\omega}$ given in Section~\ref{sec:simplecondition}.

\section{Other proofs}\label{Sec:AutresPreuves}

\subsection{Proof of Theorem~\ref{Thm:main2}}\label{sec:preuveL2}
Setting $V(x) = e^{c|x|^2/4}$, Assumption~\ref{Assu1} classically implies that $LV\leqslant -aV + C$   for some constants $C,a>0$, and that $V\in L^2 (\mu)$. Moreover, by standard elliptic theory (see e.g. \cite[Theorem 0.5 and Condition 0.24.A1]{DynkinVol2}), the process \eqref{eq:EDS} admits a continuous positive transition kernel, hence, for all compact set $\mathcal K\subset \R^d$ and all $t>0$,  there exists $\eta>0$ such that $\inf_{x\in\mathcal K} P_t f(x) \geqslant \int_{\mathcal K} f(z)\dd z$ for all positive $f$. From \cite{HairerMattingly2008}, we get that $P_t f(x) \rightarrow \mu f$ as $t\rightarrow \infty$ for all $f\in \A$ (since $f/V$ is bounded if $f\in\A$). 
 For $f\in\A$, $x\in\R^d$ and $t\geqslant0$,
\begin{eqnarray*}
P_t f^2(x) - (P_t f(x))^2  & = &  \int_0^t \partial_s \po P_s(P_{t-s} f)^2(x)\pf \dd s\\
& = &  2 T \int_0^t P_s|\na P_{t-s} f|^2 (x)\dd s \\
& \leqslant & 2T P_t|\na f|^2 (x) M^2 \int_0^t e^{-2\lambda s} \dd s 
\end{eqnarray*}
where we used the equivalence of Wasserstein and gradient contractions \nv{of Proposition~\ref{prop:kuwada}}. Letting $t\rightarrow \infty$ in the previous equality yields the Poincaré inequality (for all $f\in\A$, and then for all $f\in L^2(\mu)$ by density).

From the    Lumer-Philips Theorem \cite[Chapter IX, p.250]{Yosida}, the Poincaré inequality is equivalent to
\[\forall f\in L^2(\mu)\,, \forall t\geqslant 0\qquad \|P_tf -\mu f\|_{L^2(\mu)} \leqslant e^{-tT/C_P}\|f -\mu f\|_{L^2(\mu)} \,.\]
Besides, for $t>0$ and $f\in\A$, we can also bound
\[P_t f^2 - (P_t f)^2  =  2 T \int_0^t P_s|\na P_{t-s} f|^2 \dd s \geqslant 2 T |\na P_{t} f|^2 M^{-2}\int_0^t e^{2\lambda s} \dd s\,. \]
Integrating with respect to $\mu$ and applying this with $f$ replaced by $f-\mu f$ yields
\[\|\na P_t f\|_{L^2(\mu)}^2 \leqslant \frac{M^2\lambda}{T(e^{2\lambda t}-1)} \po \|f-\mu f\|_{L^2(\mu)}^2- \|P_t f-\mu f\|_{L^2(\mu)}^2\pf \,.\]
Together with the Poincaré inequality, this means that,  for all $t\geqslant 0$, 
\[\|P_t f - \mu f\|_{L^2(\mu)}^2 \leqslant  \frac{C_PM^2\lambda}{T(e^{2\lambda t}-1)}  \po \|f-\mu f\|_{L^2(\mu)}^2- \|P_t f-\mu f\|_{L^2(\mu)}^2\pf\,,\]
i.e.
\[\|P_t f - \mu f\|_{L^2(\mu)}^2 \leqslant  \po 1+ \nv{\frac{T(e^{2\lambda t}-1)}{C_PM^2\lambda }}\pf^{-1}\|f-\mu f\|_{L^2(\mu)}^2\,.\]

Now, in the reversible case where $b=-\na U$, \nv{since
\[- \frac1t \ln  \po 1+ \nv{\frac{T(e^{2\lambda t}-1)}{C_PM^2\lambda }}\pf^{-1} = 2\lambda\,,\]
}
 this implies (see e.g. \cite[\nv{Lemma 2.14}]{CattiauxGuillinPAZ}) that, in fact,
\[\|P_t f - \mu f\|_{L^2(\mu)}^2 \leqslant  e^{-2\lambda t} \|f-\mu f\|_{L^2(\mu)}^2\,,\]
and thus  $C_P(\mu)\leqslant T/\lambda$. 

Finally, the first inequality in \eqref{eq:thm2Wang} is simply the Pinsker inequality, and the entropy/$\mathcal W_2$ regularisation is proven in \cite{RocknerWang}. It is assumed in the latter work that $b$ is Lipschitz, but it is not used in this part of the proof. We briefly recall the proof for completeness. Remark that if Assumption~\ref{Assu1} holds with $K=0$ then it also holds with all $K>0$ and thus it is sufficient to treat the case $K>0$. From \cite[Theorem 3.3]{Wang}, under Assumption~\ref{Assu1}, for all $t>0$ all positive $f$ and all $x,y\in\R^d$,
\[P_t \ln f(x) \leqslant \ln P_t f(y) + \frac{K|x-y|^2}{2T(1-e^{-2Kt})} \,.\]
Denoting by $P_t^*$ the dual of $P_t$ in $L^2(\mu)$, we apply the previous inequality with $f$ replaced by $P_t^* f$ for some positive $f$ with $\mu f= 1$ and integrate with a coupling measure $\pi \in \Pi( f\mu, \mu)$ to get
\begin{eqnarray*}
\int_{\R^d} P_t^* f(x)  \ln P_t^* f(x)   \mu(\dd x)  & =& \int_{\R^d} P_t\po \ln P_t^* f(x) \pf f(x) \mu(\dd x)\\
& \leqslant & \int_{\R^d} \ln P_t P_t^* f(y) \mu(\dd y) + \frac{K}{2T(1-e^{-2Kt})} \int_{\R^{2d}} |x-y|^2 \pi(\dd x,\dd y) \,.
\end{eqnarray*}
Using that $\mu$ is invariant by $P_t P_t^*$ and Jensen's inequality, 
\[\mu(\ln P_tP_t^* f) \leqslant \ln \mu(P_tP_t^* f) = \ln \mu f = 0\,,\]
and taking the infimum over $\pi$ concludes the proof of \eqref{eq:thm2Wang} with $J(t) = K/(1-e^{-2Kt})$. Then, using the $\mathcal W_2$ contraction, for all $s\in[0,t)$,
\[\nu P_{t} \po \ln \frac{\nu P_{t}}{\mu}\pf \leqslant \frac{K}{2T(1-e^{-2K(t-s)})} M^2 e^{-2\lambda s} \mathcal W_2(\nu,\mu)\,.\] 
\nv{The minimum of $s\mapsto e^{-2\lambda s} /(1-e^{-2K(t-s)})$ for  $s \leqslant t$ for a fixed $t>0$ is attained at  $s =  s_*: = t - \ln \po 1+K/\lambda\pf/(2K)$. When $s_*\geqslant 0$, the proof is concluded by taking $s=s_*$ in the previous bound, since 
\begin{eqnarray*}
\frac{K}{2T(1-e^{-2K(t-s_*)})} M^2 e^{-2\lambda s_*} &=&  \frac{M^2 K \po 1+K/\lambda\pf^{\lambda/K}e^{-2\lambda t} }{2T\po 1-(1+K/\lambda)^{-1}\pf }   \\
&=& \frac{M^2}{2T} (K+\lambda) \po 1+K/\lambda\pf^{\lambda/K}e^{-2\lambda t}\,.
\end{eqnarray*}
}

\subsection{Degenerate convex potential}\label{Sec:convexDegenere}

Let $U(x) = |x|^\beta/\beta$ for some $\beta>2$. For $T>0$, let $\mu_T$ be the probability law on $\R^d$ with density proportional to $\exp(-U/T)$. 
 If $X$ is a random variable with law $\mu_1$, then $T^{1/\beta}X$ is distributed according to $\mu_T$. 
 The scaling properties of the Poincaré inequality \nv{imply} that $C_P(\mu_T)= T^{2/\beta} C_P(\mu_1)$. For  $b = -\na U$, $\nv{D b}(0)=0$ and, for $x\neq 0$,
\[u\cdot \nv{D b}(x) u = -u\cdot \na^2 U(x) u  = -|x|^{\beta-2} \co \frac{\beta-2}\beta \frac{|x\cdot u|^2}{|x|^2} + |u|^2 \cf \leqslant - |x|^{\beta-2}  |u|^2\,.\]
 Then, for all $x\neq 0$ and all $y\neq x$,
\begin{eqnarray*}
\frac{(x-y)\cdot \po b(x)-b(y)\pf }{|x-y|^2} &=&  -\int_0^1  \frac{x-y}{|x-y|}  \cdot \na^2U(x+t(x-y)) \frac{x-y}{|x-y|}  \dd t\\
& \leqslant &  -  \int_0^1  |x+t(x-y)|^{\beta-2}  \dd t\\ 
&\leqslant &  -  |x|^{\beta-2} \inf_{\nv{r}\in\R } \int_0^1 |1+t\nv{r}|^{\beta-2} \dd t\\
& = &  - |x|^{\beta-2} \inf_{\nv{r}\geqslant 0 } \int_0^1 |1-t\nv{r}|^{\beta-2} \dd t\\
& = &  - |x|^{\beta-2} \inf_{\nv{r}\geqslant 0 } \frac1{\nv{r}} \int_{1-\nv{r}}^1 |s|^{\beta-2} \dd s\,.
\end{eqnarray*}
The term in the $\inf$ is non-increasing for $\nv{r}\leqslant 1$, non-decreasing for   $\nv{r}\geqslant 2$ and, for $\nv{r}\in[1,2]$, 
\[\frac1{\nv{r}}\int_{1-\nv{r}}^1 |s|^{\beta-2} \dd s \geqslant \frac12 \int_0^1 |s|^{\beta-2} \dd s = \frac{1}{2(\beta-1)} \,.\]
As a conclusion, for all $x\in\R^d$
\[k(x) \geqslant \frac{1}{2(\beta-1)} |x|^{\beta-2}\,,\]
which means that, for all $r>0$, Assumption~\ref{Assu1} holds with $K=0$, $R=r$ and $c=c(r):=  r^{\beta-2}/(2\beta-2)$. Besides, $-x\cdot b(x) =  |x|^\beta $ for all $x\in\R^d$ and thus Theorem~\ref{Thm:main}  applied with $\alpha=2$ and
\[T = T(r) :=  \frac{ 2^{2/d-1} r^2 c(r)+2(2^{1/d}r)^\beta}{d}  = \frac{ r^{\beta}  }{ d} \po \frac{1}{2^{2-2/d} (\beta-1)}   + 2^{\beta/d+1}\pf\leqslant \frac{r^\beta}{d}\po 1+ 2^{\beta/d+1} \pf \,,\]
yields (according to Theorem~\ref{Thm:main2} since we are in the reversible case)
\[C_P(\mu_{T(r)}) \leqslant T(r)\frac{4}{c(r)}\]
hence
\[C_P(\mu_1)  \leqslant  \po T(r)\pf^{1-2/\beta}\frac{4}{c(r)} \leqslant \frac{ 8  (\beta-1) \po 1+ 2^{\beta/d+1} \pf^{1-2/\beta}}{d^{1-2/\beta}}\,.\]
Notice that, as expected due to the homogeneity of the problem, the powers of $r$ have disappeared.

\subsection{The reversible case}\label{Sec:reversible}

If $b=-\na U$ and Assumption~\ref{Assu1} holds, let us check what explicit estimate can be obtained by the Holley-Stroock perturbation argument together with the Bakry-Emery curvature one. As  in Section~\ref{sec:localCondition}, we use that Assumption~\ref{Assu1} implies that $\na^2 U(x) \geqslant c\1_{|x|\geqslant R}-K\1_{|x|<R}$. Consider $V(x)= - g(|x|^2)$ for some $g\in\mathcal C^2$ to be chosen. Assume that $g$ is bounded and that $\na^2 (U+V)(x) \geqslant c/2$ for all $x\in\R^d$. Then the Bakry-Emery criterion states that $\tilde \mu \propto \exp(-(U+V)/T)$ satisfies a Poincaré inequality with constant $2T/c$, and by bounded perturbation we get that $\mu\propto \exp(-U/T)$ satisfies a Poincaré inequality with constant $2T e^{(\max g - \min g)/T}/c$.

It remains to choose $g$. We \nv{choose} $g$ to be non-increasing and convex, so that we bound
\[\na^2 V(x) = - 4 x x^T g''(|x|^2) - 2 I_d g'(|x|^2) \geqslant -4|x|^2 g''(|x|^2) - 2 g'(|x|^2) \,.\]
As in Section~\ref{sec:simplecondition}, we simply take $g'$ as (a $\mathcal C^1$ non-decreasing approximation of) the continuous solution of $g'(0)=-K/2-c/4$ and
\[r g''(r) +  \frac{1}2 g'(r) =  \left\{\begin{array}{ll}
-(K/4+c/8)  & \text{for }r<R^2\\
c/8 & \text{for }r\in(R^2,R_*^2)\\
0 & \text{for }r>R_*^2
\end{array}\right.\]
with $R_* =2R(1+K/c)$.  In other words, $g$ is exactly such as constructed in Section~\ref{sec:simplecondition} with $d=1$. We end up with
\[\max g - \min g \leqslant  (2K+c)R^2(1+K/c)^2\,. \]
As a consequence,
\[C_P(\mu) \leqslant  \frac{2T}{c}\exp\po \frac{(2K+c)R^2(1+K/c)^2}{T}\pf \,.\]

\subsection{\nv{Perturbation} of the drift}\label{sec:perturbation}

This section is devoted to the proof of Proposition~\ref{prop:perturbation}.

Define $\rho$ as
\[\rho(x,y) = |x-y|^\alpha \po T+ \alpha \nv{\omega}(x)\pf\,,\qquad x,y\in\R^d \]
for some $\alpha\geqslant 2$, where $\nv{\omega}$ is a positive bounded function to be chosen. \nv{Let $(X_t)_{t\geqslant 0}$ and $(Y_t)_{t\geqslant 0}$ be the solution respectively of \eqref{eq:EDS} and \eqref{eq:EDSpertubée}} driven by $(B_t)_{t\geqslant 0}$ with some initial condition $\nu$. Then
\begin{eqnarray*}
\lefteqn{\partial_t \mathbb E \po \rho(X_t,Y_t) \pf}\\
 & =& \alpha  \mathbb E\po  |X_t-Y_t|^{\alpha-2} (X_t-Y_t)\cdot \po b(X_t) - \tilde b(Y_t,Z_t) \pf (T+\alpha \nv{\omega}(X_t)) + |X_t-Y_t|^\alpha L\nv{\omega}(X_t)\pf \\
 & \leqslant & \alpha \mathbb E \po |X_t-Y_t|^\alpha \Psi(X_t) +  |X_t-Y_t|^{\alpha-1}(T+\alpha\|\nv{\omega}\|_\infty) | b(Y_t) - \tilde b(Y_t,Z_t)|\pf 
\end{eqnarray*}
with $\Psi$ defined in \eqref{eq:Psi}. Let $\nv{\omega}$ be as defined in Section~\ref{sec:simplecondition}, and $\lambda=c/4$. Then, writing $m(t)=\mathbb E \po \rho(X_t,Y_t) \pf$ and using the Hölder inequality, 
\begin{eqnarray*}
m'(t)&  \leqslant & -\lambda \alpha m(t)
   + \alpha (T+\alpha\|\nv{\omega}\|_\infty) \po \mathbb E \po      |X_t-Y_t|^{\alpha} \pf \pf^{1-1/\alpha}\po \mathbb E \po | b(Y_t) - \tilde b(Y_t,Z_t)|^\alpha\pf\pf^{1/\alpha} \\
  &  \leqslant & -\lambda \alpha m(t) + \alpha \widetilde C(t) m(t)^{1-1/\alpha}
\end{eqnarray*}
with
\[\widetilde C(t) =  \frac{T+\alpha\|\nv{\omega}\|_\infty}{T^{1-1/\alpha}}  \po \mathbb E \po | b(Y_t) - \tilde b(Y_t,Z_t)|^\alpha\pf\pf^{1/\alpha}\,.\]
Then 
\[( m^{1/\alpha})'(t) \leqslant -\lambda  m^{1/\alpha}(t) + \widetilde C(t) \] 
and thus
\[m^{1/\alpha}(t) \leqslant e^{-\lambda t}m^{1/\alpha}(0) + \int_0^t  e^{\lambda (s-t)}\widetilde C(s)\dd s\,.\]
Using the equivalence between $\rho$ and $|x-y|^\alpha$ and taking the infimum over all coupling of the initial conditions yields
\[\mathcal W_\alpha(\nu P_t,\tilde \nu_t) \leqslant M_\alpha e^{-\lambda t} \mathcal W_\alpha(\nu ,\tilde \nu_0)+ T^{-1/\alpha} \int_0^t  e^{\lambda (s-t)}\widetilde C(s)\dd s\,,\] 
which concludes the proof.

\section{Interacting particles at high temperature}\label{sec:particules}

\subsection{General framework}
Let $\bX_t = (X_{1,t},\dots,X_{N,t})$ be a process on $(\R^d)^N$ solving
\[\forall i\in\cco 1,N\ccf\,,\qquad \dd X_{i,t} = F(X_{i,t}) \dd t + G_i(\bX_t) \dd t + \sqrt{2T}\dd B_{i,t}\]
where $(B_{1,t},\dots,B_{N,t})_{t\geqslant 0}$ are $N$ independent standard $d$-dimensional Brownian motions, $F\in\mathcal C^1(\R^d,\R^d)$, $G_i\in \mathcal C^1(\R^{dN},\R^{d})$. In other words, $\bX$ solves \eqref{eq:EDS} with a drift $b$ whose $i^{th}$ $d$-dimensional component is  $b_i(\bx)= F(x_i)+G_i(\bx)$ for all $i\in\cco 1,N\ccf$. \nv{Write $G(\bx)=(G_1(\bx),\dots,G_N(\bx))\in (\R^d)^{N}$.}

\begin{thm}\label{thm:particules}
Assume that there exist $c>0$, $a<c$ and $C_F,C_G,R,M_G\geqslant 0$ such that 
\begin{equation}\label{eq:F}
 (x-y)\cdot (F(x)-F(y)) \leqslant \left\{\begin{array}{ll}
C_F|x-y|^2 & \text{for all } x,y\in\R^d\\
-c|x-y|^2 & \text{for all } x,y\in\R^d \text{ with } |x|\geqslant R 
\end{array}\right.
\end{equation}
and\nv{, for all $\bx,\by\in \R^{dN}$}
\begin{equation}\label{eq:revision}
\sum_{i=1}^N |x_i-y_i||G_i(\bx) - G_i(\by)| \leqslant C_G |\bx-\by|^2 \,,\qquad 
 (\bx-\by)\cdot (G(\bx)-G(\by))  \leqslant  a|\bx-\by|^2
\end{equation}
\nv{and $|G_i(\bx)|\1_{|x_i|\leqslant \nv{R_*}}\leqslant M_G$ for all $i\in\cco 1,N\ccf$ where $R_* = R\po 2+2(C_F+a)/(c-a)\pf^{1/(2d)}$.}

 Writing $K_* = C_F+(c+a)/2$,  assume furthermore that
 \[T \geqslant
 T_0   = \frac{K_*}{d(c-a)} \po  R_*^2\po C_F+C_G\pf + 2 \sup\{-F(x)\cdot x, |x|\leqslant R_*\} +4 M_GR_* \pf \,. \]
 Then the semi-group $P_t$ associated to the process $\bX$ satisfies the $\mathcal W_2$ contraction \eqref{eq:contractW} (with $\alpha=2$) with
 \[M= \sqrt{1+2K_* R_*^2/(Td)}\qquad \lambda = \frac{c-a}{4+8K_* R_*^2/(Td)}\,.\]
\end{thm}

\nv{Notice that the first part of \eqref{eq:revision} implies the second one with $a = C_G$, but in many cases we can have $a<C_G$ (possibly $a\leqslant 0$, see next section) and the result is much more sensible to the value of $a$ (in particular with the condition $a<c$) than to the value of $C_G$. }

\begin{proof}
Consider
\[\rho(\bx,\by) = \sum_{i=1}^N |x_i-y_i|^{2} \po T +  \nv{\omega}(x_i)+  \nv{\omega}(y_i)\pf \]
where $\nv{\omega}$ is a positive function to be chosen. As in Section~\ref{sec:simplecondition}, writing $\mathcal L_s$ the generator on $\R^{dN}\times\R^{dN}$ of a parallel coupling of two processes, we consider separately the leading terms with respect to $T$ and the rest in
\begin{eqnarray*}
\mathcal L_s \rho(\bx,\by) & = & TA+B  
\end{eqnarray*}
with
\begin{eqnarray*}
A &=&  \sum_{i=1}^N \co 2(x_i-y_i) \po F(x_i)+G_i(\bx)-F(y_i)-G_i(\by)\pf + |x_i-y_i|^2 \po \Delta \nv{\omega}(x_i)+\Delta\nv{\omega}(y_i)\pf\cf \\
& \leqslant & \sum_{i=1}^N |x_i-y_i|^2 \po 2a-k_F(x_i)-k_F(y_i) + \Delta \nv{\omega}(x_i)+\Delta\nv{\omega}(y_i)\pf  
\end{eqnarray*}
where $k_F(x_i)=c$ if $|x_i|\geqslant R$ and $k_F(x_i)=-C_F$ otherwise and, \nv{assuming that $\omega$ is constant outside the ball $\{|x|\leqslant R_*\}$,} 
\begin{eqnarray*}
B &=&  \sum_{i=1}^N  \co 2 (x_i-y_i)\cdot \po b_i(\bx)-b_i(\by)\pf \po  \nv{\omega}(x_i)+\nv{\omega}(y_i)\pf + |x_i-y_i|^2\po b_i(\bx)\cdot \na \nv{\omega}(x_i)+b_i(\by)\cdot \na \nv{\omega}(y_i)\pf\cf\\
& \leqslant & \po 4\|\nv{\omega}\|_\infty \po C_F+C_G\pf + 2\sup\{F\cdot\na \nv{\omega}\} + 2M_G\|\na\nv{\omega}\|_\infty\pf |\bx-\by|^2 \,.
\end{eqnarray*}
We take $\nv{\omega}$ as in Section~\ref{sec:simplecondition} but with $K,c$ replaced respectively by $C_F+a$ and $c-a$ \nv{(in particular $\omega$ is indeed constant outside the ball $\{|x|\leqslant R_*\}$)}. The following holds:
\begin{eqnarray*}
\forall x\in\R^d,\qquad a - k_F(x_i) + \Delta \nv{\omega}(x_i) & \leqslant &  - \frac{c-a}{2}\\
\|\nv{\omega}\|_\infty &\leqslant& \frac{K_* R_*^2}{d}  \\
\sup\{F\cdot\na \nv{\omega}\} &\leqslant&  \frac{K_*}{d} \sup\{-F(x)\cdot x, |x|\leqslant R_*\}\\
\|\na \nv{\omega}\|_\infty &\leqslant & 2R_* \frac{K_*}{d}\,.
\end{eqnarray*} 
Hence, provided $T\geqslant T_0$, the previous bounds yield  
\[
\mathcal L_s \rho(\bx,\by) \leqslant -\frac{c-a}2 T|\bx-\by|^2  \leqslant -\frac{c-a}{2+4\|\nv{\omega}\|_\infty/T} \rho(\bx,\by)\,.
\]
The conclusion is now similar to the end of the proof of Theorem~\ref{Thm:main}, using that
\[|\bx-\by|^2 \leqslant \frac1T \rho(\bx,\by) \leqslant \po 1+\frac{2\|\nv{\omega}\|_\infty}T\pf\|\bx-\by|^2\,.\]

\end{proof}

The point of Theorem~\ref{thm:particules} is that if all the constants in the assumption are independent from $N$, then so are $T_0,\lambda$ and $M$. In particular, from Theorem~\ref{Thm:main2}, we get for the invariant measure of the process a Poincaré inequality independent from $N$ (for $T$ large enough). The restriction to a sufficiently high temperature is very natural for interacting particles systems where phase transitions are expected in the behaviour of the Poincaré inequality at low temperature \cite{Tugaut2014}.

A Poincaré constant uniform in $N$ for $T$ large enough is established in \cite{GuillinWuZhang} in a reversible framework with an explicit invariant measure. Although Theorem~\ref{thm:particules} does not require reversibility, on the other hand it needs the interaction force $G$ to be bounded, which is not the case in \cite{GuillinWuZhang} and is a restrictive condition. It is however satisfied in many cases of interest, for instance in adaptive algorithms such as studied in \cite{Clubdes5}, a typical choice is 
\[G_i(\bx) = \frac1N \sum_{i=1}^N \na W(x_i-x_j)\,,\qquad W(x) = e^{-|x|^2}\,,\]
which induces a local repulsion of particles, enhancing the exploration of the state space.

More generally, assume that there exist a graph on $\cco 1,N\ccf$ of degree $D$  and a bounded and Lipschitz function $H\in \mathcal C^1(\R^d\times\R^d,\R^d)$ such that
\[G_i(\bx) = \frac{1}{D}\sum_{j\sim i} H(x_i,x_j)\]
where $i\sim j$ means that $(i,j)$ is an edge of the graph. This is the case for mean field interaction (with the complete graph and $D=N$) or for interaction with closest neighbors in $(\Z/n\Z)^k$ or $\cco 1,n\ccf^k$ (with $i\sim j$ if $|i-j|=1$, $D=2k$). 
 Then $G$ satisfies the assumptions of Theorem~\ref{thm:particules} with $M_G,C_G,a$ which only depend on $H$ (and thus not on the number of particles). For instance, in the particular (reversible) case where the forces are the gradients of some potentials, we get the following corollary of Theorem~\ref{thm:particules}.
 
 \begin{prop}

 Let $V\in\mathcal C^1(\R^d)$, $W\in\mathcal C^1(\R^d\times\R^d)$. Assume that $\na W$ is bounded and Lipschitz and that there exists $c>0$ and $a<c$ such that $\na^2 V \geqslant c>0$ outside a compact set and  $\na^2 W \geqslant -a/2$. Then there exist $T_0,C> 0$ such that the following holds. For all $T\geqslant T_0$, all $N\in\N$,  all graph on $\cco 1,N\ccf$, denoting by $D$ the degree of the graph and considering on $\R^{dN}$ the potential
 \[U(\bx) = \sum_{i=1}^N V(x_i) + \frac{1}{D}\sum_{i=1}^N\sum_{j\sim i} W(x_i-x_j)\,, \]
 then $e^{-U/T}$ is integrable and the probability measure proportional to this density satisfies a Poincaré inequality with constant $C_P\leqslant CT$. 

 \end{prop} 
 
 This is a result in the spirit of \cite[Theorem 1]{GuillinWuZhang}.
 
 \subsection{The mean-field case and propagation of chaos}
 
Let us now focus on the case of mean field interactions. More precisely, we work under the following condition:

\begin{assu}\label{AssuChaos}
The drift $b$ on $\R^{dN}$ is of the form $b_i(\bx) = F(x_i) + \frac1N \sum_{i=1}^N H(x_i,x_j)$ where  $F\in\mathcal C^1 (\R^d,\R^d)$ and $H\in\mathcal C^1(\R^d\times\R^d,\R^d)$. Moreover, there exist $c>0$, $a<c$ and $C_F,R,C_G,M_G\geqslant 0$ such that $F$ satisfies \eqref{eq:F}, $H$ is $2C_H/3$-Lipschitz continuous and for all $x,y,x',y'\in\R^d$, $\1_{|x|\leqslant R} H(x,y)\leqslant M_H$ and
\begin{equation}\label{eq:conditionH}
 (x-y)\cdot (H(x,x')-H(y,y')) + (x'-y')\cdot (H(x',x)-H(y',y)) \leqslant a\po |x-y|^2 + |x'-y'|^2\pf\,.
\end{equation}
Finally, $T\geqslant T_0$, where $T_0$ is given in Theorem~\ref{thm:particules}. 
\end{assu}

It is straightforward to check that this condition implies the assumptions of Theorem~\ref{thm:particules}. As soon as $H$ is $2C_H/3$-Lipschitz, \eqref{eq:conditionH} holds with $a=C_H$, and the condition $a<c$ is then satisfied if the interaction is sufficiently small. However, in some cases, $a$ may be smaller than $C_H$, in particular, in the usual case where $H(x,y)=\tilde H(x-y)=-\tilde H(y-x)$ for some $\tilde H$, the condition~\eqref{eq:conditionH} reads
\[ (x-y-x'+y')\cdot (\tilde H(x-x')-\tilde H(y-y'))  \leqslant a\po |x-y|^2 + |x'-y'|^2\pf\,,\]
and this holds with $a=0$ if $x\cdot \tilde H(x) \leqslant 0$ for all $x\in\R^d$. This is the case for instance for $\tilde H(x) = -\na W(x)$ with $W(x) = \gamma \sqrt{1+|x|^2}$, for any $\gamma\geqslant 0$. Since $\na W$ is bounded, in this case, Theorem~\ref{thm:particules} applies whatever the value of $\gamma$, i.e. even if the interaction force is not small with respect to the confining force (however, as $\gamma$ increases, so does the temperature $T_0$).

As $N\rightarrow +\infty$, according to the propagation of chaos phenomenon, it is well-known  that two given particles of the system behave like independent McKean-Vlasov processes solving
\begin{equation}\label{eq:EDSMcKV}
\dd \bar X_t = F(\bar X_t) \dd t + \int_{\R^d} H(\bar X_t,z) \rho_t(\dd z) \dd t + \sqrt{2T} \dd B_t\,,\qquad \bar \nu_t = \mathcal Law(\bar X_t)\,.
\end{equation}
In other words, $\bar \nu_t$ solves the non-linear equation
\begin{equation}\label{eq:McKeanVlasov}
\partial_t \bar \nu_t = \na\cdot \po T\na\bar \nu_t - \po F+ (H\ast \bar \nu_t)\pf\bar \nu_t\pf \,,
\end{equation}
where $ H\ast\bar  \nu (x) = \int_{\R^d} H(x,y)\bar  \nu(\dd y)$. The existence, uniqueness of the process \eqref{eq:EDSMcKV} and of the solution of the equation \eqref{eq:McKeanVlasov}, together with time-dependent propagation of chaos estimates, follow from standard arguments \cite{Sznitman,Meleard} for initial conditions $\bar\nu_0$ in $\mathcal P_2(\R^d)$ the set of probability measures on $\R^d$ with finite second moment.

With a $\mathcal W_2$ Wasserstein contraction such as given by Theorem~\ref{thm:particules} at hand, it is straightforward to obtain time-uniform propagation of chaos and a Waserstein contraction for the limit equation.

\begin{thm}\label{thm:champmoyen}
Under Assumption~\ref{AssuChaos}, there exist (explicit) constants $\alpha,\beta>0$ such that the following holds. For $N\in\N$, let $P_t^N$ be the semi-group associated to $L= b\cdot\na + T\Delta$ on $\R^{dN}$ and let $\bar \nu_t$ be a solution of  \eqref{eq:McKeanVlasov}. Then, for all $t\geqslant 0$,
\[\mathcal W_2 \po \nu P_t^N , \bar  \nu_t^{\otimes N}\pf \leqslant \ M e^{-\lambda t} \mathcal W_2 \po \nu ,  \bar \nu_0^{\otimes N}\pf + \alpha e^{-\frac{c-a}{2}t}\sqrt{\int_{\R^d} |y|^2 \bar \nu_0(\dd y)} + \beta \,, \]
where $M$ and $\lambda$ are given in Theorem~\ref{thm:particules}.
\end{thm}
\begin{proof}
The proof is essentially the same as the proof of Proposition~\ref{prop:perturbation}, except that we consider a \nv{cost} $\rho$ as in the proof of Theorem~\ref{thm:particules}, namely
\[\rho(\bx,\by) = \sum_{i=1}^N |x_i-y_i|^2 \po T+2\nv{\omega}(x_i)\pf \,.\]
Let $(\bX_t)_{t\geqslant 0}$ be a system of particles with drift $b$ and initial condition $\nu$ and  $\bY=(Y_1,\dots,Y_N)$ be solutions of \eqref{eq:EDSMcKV} (with the same Brownian motions as $\bX$) with initial condition  $ \nu_0^{\otimes N}$. In particular, $\bY_t \sim  \nv{\bar \nu}_t^{\otimes N}$ for all $t\geqslant 0$. As in the proof of Proposition~\ref{prop:perturbation}, writing $m(t) = \mathbb E\po \rho(\bX_t,\bY_t)\pf$, we get
\begin{equation}\label{eq:demochaos}
m'(t) \leqslant -2\lambda m(t) + 2 (T+2\|\nv{\omega}\|_\infty) \sum_{i=1}^N \sqrt{ \mathbb E \po |X_{i,t} - Y_{i,t}|^2\pf \hat C(t)  }
\end{equation}
where $\lambda$ is given in Theorem~\ref{thm:particules} and, using that the $Y_j$'s all have the same law,
\begin{eqnarray*}
\hat C(t)& =& \mathbb E\po \left|\frac1N\sum_{j=1}^N H(Y_{1,t},Y_{j,t}) - H\ast \bar \nu_t(Y_{1,t})\right|^2\pf\,.
\end{eqnarray*}
 Developing the square and using that the variables $A_j:= H(Y_{1,t},Y_{j,t}) - H\ast \bar \nu_t(Y_{1,t})$ for $j\neq 1$  are independent and centered, we get
\begin{eqnarray*}
\hat C(t)& \leqslant & \frac1{N^2}  \mathbb E \po \sum_{j=1}^N   |A_j|^2 + A_1\cdot \sum_{j\neq 1} A_j \pf    \\
& \leqslant & \frac1{N}  \mathbb E \po |A_1|^2 + |A_2|^2 \pf\,. 
\end{eqnarray*}
Then we bound
\begin{eqnarray*}
\mathbb E \po |A_j|^2 \pf &= & \mathbb E\po \left| \int_{\R^d} (H(Y_{1,t},Y_{j,t}) - H(Y_{1,t},y) ) \bar \nu_t(\dd y)\right|^2 \pf \\
& \leqslant & L_H^2 \mathbb E \po \int_{\R^d} \left| Y_{j,t}-y \right|^2  \bar \nu_t(\dd y)  \pf \ = \ 2 L_H^2 \po  \int_{\R^d} |y|^2 \nv{\bar \nu}_t(\dd y) - \left|\int_{\R^d} y \nv{\bar \nu}_t(\dd y)\right|^2\pf 
\end{eqnarray*}
and thus, 
\[\hat C(t) \leqslant \frac{4L_H^2}{N} \int_{\R^d} |y|^2 \nv{\bar \nu}_t(\dd y)\,.\]
Under Assumption~\ref{AssuChaos}, for all $y,y'\in\R^d$,
\begin{eqnarray*}
y\cdot F(y)& \leqslant&  |y||F(0)| - c|y|^2 +(C_F+c)R^2 \\
y\cdot  H(y,y') + y' \cdot  H(y',y)  &\leqslant &a \po |y|^2+|y'|^2\pf + \po|y|+|y'|\pf |H(0,0)|
\end{eqnarray*}
from which
\begin{eqnarray*}
\partial_t   \int_{\R^d} |y|^2 \nv{\bar \nu}_t(\dd y) & = &   2\int_{\R^d\times\R^d} \co  y \cdot \po F(y) + H(y,y') \pf + T d \cf \nv{\bar \nu}_t(\dd y)\nv{\bar \nu} _t(\dd y')\\
 & \leqslant & - (c-a) \int_{\R^d} |y|^2 \nv{\bar \nu}_t(\dd y)  +Q
\end{eqnarray*}
with
\[Q = Td +  2(C_F+c)R^2 + \frac{1}{c-a}\po  |H(0,0)|+|F(0)|\pf^2\,.\]
Integrating in time,
\[\int_{\R^d} |y|^2 \nv{\bar \nu}_t(\dd y) \leqslant e^{-(c-a)t} \int_{\R^d} |y|^2 \nv{\bar \nu}_0(\dd y) +   Q\]
Going back to \eqref{eq:demochaos},
\begin{eqnarray*}
m'(t) &\leqslant& -2\lambda m(t) + 2 (T+2\|\nv{\omega}\|_\infty) \sqrt{ N\sum_{i=1}^N  \mathbb E \po |X_{i,t} - Y_{i,t}|^2\pf \hat C(t)  }\\
& \leqslant & -2\lambda m(t) + 4 \frac{L_H(T+2\|\nv{\omega}\|_\infty)}{\sqrt T} \sqrt{m(t)\po e^{-(c-a)t}  \int_{\R^d} |y|^2 \nv{\bar \nu}_0(\dd y) +  Q\pf }\\
& := & -2\lambda m(t) + 2\sqrt{m(t) \po e^{-(c-a)t}A^2+B^2\pf }\,,
\end{eqnarray*}
hence
\[ (\sqrt{m})'(t) \leqslant -\lambda \sqrt{m(t)} + \sqrt{ e^{-(c-a)t}A^2+B^2 }  \leqslant -\lambda \sqrt{m(t)} + e^{-(c-a)t/2}A+B \,.\] 
Integrating in time (and noticing that $\lambda \leqslant (c-a)/4$),
\[\sqrt{m(t)} \leqslant e^{-\lambda t} \sqrt{m(0)} + \int_0^t e^{\lambda(s-t)}\po e^{-(c-a)s/2} A+ B\pf\dd s  \leqslant e^{-\lambda t} \sqrt{m(0)}  + \frac{e^{-(c-a)t/2}}{(c-a)/2-\lambda}A   +  \frac{1}{\lambda} B \,,\]
and the proof is concluded by using the equivalence between $\rho$ and the Euclidean norm and taking the infimum over the couplings of the initial distributions. 
\end{proof}

 Theorem~\ref{thm:champmoyen} has the following consequences.
 
 \begin{cor}
 Under Assumption~\ref{AssuChaos}, considering $M,\lambda,\alpha,\beta$ as in Theorem~\ref{thm:champmoyen}, then, for all $N\in \N$, $k\in\cco 1,N\ccf$ and all $\bar \nu_0,\bar \mu_0 \in \mathcal P_2(\R^d)$, the following holds. Let $(\bX_t)_{t\geqslant 0}$ be a system of $N$ interacting particles on $\R^{dN}$ with drift $b$ and with initial condition $\bar \nu_0^{\otimes N}$ and  denote by $\nu_t^{k,N}$ the law of $(X_{1,t},\dots, X_{k,t})$. Let $\bar \nu_t,\bar \mu_t$ be the solutions of \eqref{eq:McKeanVlasov} with respective initial conditions $\bar\nu_0,\bar\mu_0$. Then, for all $t\geqslant 0$,
 \begin{eqnarray*}
\mathcal W_2\po \nu_t^{k,N},\bar \nu_t^{\otimes k}\pf &\leqslant& \sqrt{\frac{k}{N}} \po \alpha e^{-\frac{c-a}{2}t}\sqrt{\int_{\R^d} |y|^2  \bar \nu_0(\dd y)} + \beta\pf  \\
 \mathcal W_2\po \bar \mu_t,\bar \nu_t\pf &\leqslant&  M e^{-\lambda t} \mathcal W_2\po \bar \mu_0,\bar \nu_0\pf  
 \end{eqnarray*}
 and there exists a unique stationary solution to  \eqref{eq:McKeanVlasov} in $\mathcal P_2(\R^d)$.
 
Moreover, if $\int_{\R^d} |y|^5 \bar \nu_0(\dd y)<+\infty$ then  for all $t\geqslant 0$,
\[ \mathbb E \po \mathcal W_2^2 \po \frac1N  \sum_{i=1}^N \delta_{X_{i,t}}, \bar\nu_t\pf \pf \leqslant C'\alpha(N)\po  e^{-(c-a)t}\int_{\R^d} |y|^5 \bar \nu_0(\dd y) + 1\pf^{2/5}\,. \] 
where $C'$ is a constant which depends only on $d$ and  on the parameters of Assumption~\ref{AssuChaos}, and 
\[\alpha(N) = \left\{\begin{array}{ll}
N^{-1/2} & \text{if } d<4\\
N^{-1/2}\ln(1+N) & \text{if }d=4\\
N^{-2/d} & \text{if }d>4.
\end{array}\right.\]
 
 \end{cor}

In the last claim we assumed a finite $5^{th}$ moment to get a simple statement but, as can be seen in the proof and from the results of \cite{FournierGuillin}, a similar result would hold assuming only a $q^{th}$ finite moment for any $q>2$. If only a second moment is available, we still get a similar result if $\mathcal W_2^2$ is replaced by $\mathcal W_p^p$ for any $p<2$.

\begin{proof}
Using the interchangeability of particles and that any coupling of $\nu P_t$ and $\bar \nu_t^{\otimes N}$ gives a coupling of the $k$ first particles immediatly yield
\[\mathcal W_2\po \nu_t^{k,N},\bar \nu_t^{\otimes k}\pf \leqslant \sqrt{\frac kN}\mathcal W_2\po \bar \nu_0^{\otimes N} P_t^N,\bar \nu_t^{\otimes N}\pf\,. \]
The first claim then follows from Theorem~\ref{thm:champmoyen}. By the same argument, denoting by $\mu_t^{1,N}$ the first $d$-dimensional marginal of $\bar \mu_0^{\otimes N} P_t^N$, we get
\begin{equation}\label{eq:propchaosW}
\mathcal W_2\po \nu_t^{1,N},\mu_t^{1,N}\pf \leqslant \frac1{\sqrt N} \mathcal W_2\po \bar \nu_0^{\otimes N} P_t^N,\bar \mu_0^{\otimes N} P_t^N\pf \leqslant   \frac1{\sqrt N}   M e^{-\lambda t}\mathcal W_2\po \bar \nu_0^{\otimes N} ,\bar \mu_0^{\otimes N}\pf\,,
\end{equation}
using Theorem~\ref{thm:particules}. Considering a coupling of $\bar \nu_0^{\otimes N}$ and $\bar \mu_0^{\otimes N}$ of the form $\pi_0^{\otimes N}$ where $\pi_0 \in \Pi(\bar \nu_0,\bar \mu_0)$ and then taking the infimum over $\pi_0$ yields
\[\frac{1}{\sqrt N}\mathcal W_2\po\bar  \nu_0^{\otimes N} ,\bar \mu_0^{\otimes N}\pf \leqslant \mathcal W_2\po \bar \nu_0 ,\bar \mu_0\pf\,,\]
and the second claim is thus obtained by letting $N$ go to infinity in \eqref{eq:propchaosW}.

As a consequence, for $t$ large enough, the function $\Phi_t:\bar \nu_0 \mapsto \bar \nu_t$, where $(\bar \nu_t)_{t\geqslant 0}$ is the solution of \eqref{eq:McKeanVlasov} with initial condition $\bar\nu_0$, is a contraction of $\mathcal P_2(\R^d)$ endowed with the $\mathcal W_2$ distance, which is complete. Hence, $\Phi_t$ admits a unique fixed point for $t$ large enough, and using that $\Phi_t \Phi_s= \Phi_s \Phi_t$ for all $s\geqslant 0$ and the uniqueness of the fixed point we get that the fixed point of $\Phi_t$ is in fact a fixed point of $\Phi_s$ for all $s\geqslant 0$, i.e. is a stationary solution of \eqref{eq:McKeanVlasov}.

For the last claim, let $(\bX_t,\bY_t)\sim \pi_t\in \Pi\po\bar\nu_0^{\otimes N} P_t,\bar \nu_t^{\otimes N}\pf$. We bound
\begin{eqnarray*}
\lefteqn{\mathbb E \po \mathcal W_2^2 \po \frac1N  \sum_{i=1}^N \delta_{X_{i,t}}, \bar\nu_t\pf \pf}\\
 & \leqslant & 2  \mathbb E \po \mathcal W_2^2 \po \frac1N  \sum_{i=1}^N \delta_{X_{i,t}}, \frac1N  \sum_{i=1}^N \delta_{Y_{i,t}} \pf\pf + 2 \mathbb E \po \mathcal W_2^2 \po \frac1N  \sum_{i=1}^N \delta_{Y_{i,t}}, \bar\nu_t\pf \pf\\
 & \leqslant & \frac{2}N \mathbb E\po |\bX_t-\bY_t|^2\pf + c_d \alpha(N)\po\int_{\R^d} |y|^5 \bar \nu_t(\dd y)\pf^{2/5}
\end{eqnarray*}
with $c_d$ a constant that depends only on $d$, where we used the coupling $(X_{J,t},Y_{J,t})$ with $J$ a random variable uniformly distributed over $\cco 1,N\ccf$ independent from $(\bX_t,\bY_t)$ to bound the first term, and \cite[Theorem 1]{FournierGuillin} for the second one. Then, reasoning exactly as in the proof of Theorem~\ref{thm:champmoyen}, we get that 
\[\int_{\R^d} |y|^5 \bar \nu_t(\dd y) \leqslant e^{-(c-a)t}\int_{\R^d} |y|^5 \bar \nu_0(\dd y) + Q'\]
for some $Q'>0$ which depend on the parameters of Assumption~\ref{Assu1}. Taking the infimum over $\pi_t\in \Pi\po\bar\nu_0^{\otimes N} P_t,\bar \nu_t^{\otimes N}\pf$ we end up with
\[
\mathbb E \po \mathcal W_2^2 \po \frac1N  \sum_{i=1}^N \delta_{X_{i,t}}, \bar\nu_t\pf \pf
  \leqslant  \frac{2}N \mathcal W_2^2\po \bar\nu_0^{\otimes N} P_t,\bar \nu_t^{\otimes N}\pf   + c_d \alpha(N)\po  e^{-(c-a)t}\int_{\R^d} |y|^5 \bar \nu_0(\dd y) + Q'\pf^{2/5}\,,
\]
and Theorem~\ref{thm:champmoyen} concludes the proof.
\end{proof}

\section{Discussion}\label{sec:discussion} \label{Sec:reflexionProof}

\subsubsection*{On the two proofs}

First, let us notice that the main ingredient of the two proofs of Theorem~\ref{Thm:main} is the very simple construction of a  weighted distance. Indeed, the weighted gradient of the second proof is
\[a(x) |\na f(x)|^2 = \lim_{r\rightarrow 0}\sup_{|y-x|\leqslant r} \frac{|f(x)-f(y)|^2}{|x-y|^2 \po a^{-1}(x)+a^{-1}(y)\pf/2 } := \lim_{r\rightarrow 0}\sup_{|y-x|\leqslant r} \frac{|f(x)-f(y)|^2}{r^2(x,y) }\,,\]
in other words $\sqrt{a}\na$  is the gradient associated to the weighted distance $r$.   Weighted distances  are a very standard tool, in particular for the study of the long-time convergence of Markov processes. However, it seems to us that the main originality of our work is that the role of the weight is exactly the converse of the usual one. Indeed, under Assumption~\ref{Assu1}, typically (see e.g. \cite{HairerMattingly2008,HMWasserstein,EberleGuillinZimmer,MonmarcheElementary,EberleMajka}), one considers \nv{costs} of the form $\rho(x,y) = \mathrm{d}(x,y)(1+V(x)+V(y))$ where $\mathrm{d}(x,y)=\1_{x\neq y}$ or $\mathrm{d}(x,y)=f(|x-y|)$ for some $f$  is the initial distance we are interested in and $V$ is a Lyapunov function (say, $V(x) = |x|^2$), which satisfies $LV\leqslant - cV$ outside a compact ball, thanks to the deterministic drift part $b\cdot \na$. This Lyapunov condition is then combined with some local information (local Poincaré inequality, local Doeblin condition, local coupling condition\dots) which is available on compact sets. In other words, the weight is used to obtain a decay outside a given compact set. On the contrary, in our case, we  take a weight of the form $V(x) = C-|x|^2$ in a compact ball, so that $LV\leqslant -cV$ in this ball thanks to the dissipative part $T\Delta$.

One of the interest of the first proof is that it gives more information than simply a contraction of the Wasserstein distance at time $t>0$: it shows that 
\[\forall t\geqslant 0\,,\qquad \mathbb E \po |X_t-Y_t|^\alpha\pf  \leqslant M_\alpha^\alpha e^{-\alpha \lambda t} \mathbb E \po |X_0-Y_0|^\alpha\pf \]
 where $(X_t,Y_t)_{t\geqslant 0}$ is the synchronous coupling of two diffusions. In particular, this is a Markovian coupling (i.e. $(X,Y)$ is a Markov process), realized with a single process for all times.
 
 On the other hand, one of the interest of the second proof is that it can be adapted to deal with quantities integrated with respect to $\mu$, which could be interesting in the perspective of proving a Poincaré inequality for non-explicit invariant measure of non-reversible diffusion processes without any restriction on $T$ but with the assumption that $\mu$ satisfies local Poincaré inequalities (which straightforwardly follow from elliptic lower and upper bounds on the transition density). Indeed, once integrated with respect to $\mu$, the computations of Section~\ref{Sec:preuveAlternative} reads
 \[\partial_t \mu\po a |\na P_t|^2\pf \leqslant -2\lambda \mu\po a |\na P_t|^2\pf\]
 and thus
 \[\mu\po a |\na P_t|^2\pf \leqslant \|a\|_\infty \|a^{-1}\|_\infty e^{-2\lambda t} \mu|\na f|^2\,,\]
 which doesn't give a $\mathcal W_2$-contraction but is sufficient to get a Poincaré inequality (see Section~\ref{sec:preuveL2}). An argument somehow in this spirit is given in \cite{BaudoinGordinaHerzog}, although in a completely different framework, i.e. with a non-elliptic hypoelliptic diffusion, a singular drift and weighted Poincaré inequalities. On the other hand, although the process is non-reversible in this case, its invariant measure is known, and thus the proof relies on the knowledge of $L^*$ the adjoint of $L$ in $L^2(\mu)$, which is unavailable if $\mu$ is unknown. Besides, in the elliptic case, if $L^*$ is known, it should be possible to adapt the arguments of \cite{CattiauxGuillin,BakryCattiauxGuillin}  to get a Poincaré inequality by considering a Lyapunov function  with respect to $L^*$ rather than $L$, as in \cite{BaudoinGordinaHerzog}.

 \subsubsection*{Flat torus}

As emphasized in the previous paragraph,  the proof relies on a synchronous coupling. Now, consider the very simple case of the Brownian motion on the torus $\T = \R/\Z$, i.e. $L=\Delta$. Considering a synchronous coupling of two such processes leads to $|X_t-Y_t|=|X_0-Y_0|$ (where we write $|x-y|$ the distance on $\T$). Hence, any $\rho:\T\times\T\rightarrow \R_+$ such that $c|x-y|^\alpha \leqslant \rho(x,y)\leqslant C|x-y|^\alpha$ for some $c,C,\alpha>0$ necessarily satisfies
\[\rho(X_t,Y_t) \geqslant \frac{c}{C}\rho(X_0,Y_0)\,.\]
  Hence, the first proof of Theorem~\ref{Thm:main} cannot apply in this case.
  
  Notice that, for compact manifolds, Poincaré inequalities can be obtained by lower and upper bounds on the transition density and then perturbation of the Lebesgue measure. See also  \cite{Wang2,Wang3} in the reversible case. Besides, notice that a weighted distance is used in \cite{Wang3}, but with the motivation of handling boundaries.

\subsubsection*{Exponential tail}

Consider on $\R$ the drift $b(x) = - U'(x)$ with $U(x)=\sqrt{1+x^2}$. It is well-known that $\mu \propto \exp(-U)$ satisfies a Poincaré inequality. However, using that $\|U'\|_\infty = 1$, and using that $\partial_t \mathbb E (X_t) = \mathbb E(b(X_t))$, we get that
\[|\mathbb E(X_t) - \mathbb E(X_0)| \leqslant t\,.\]
This clearly forbids a Wasserstein contraction \eqref{eq:contractW}, for any $\alpha\geqslant 1$. Indeed,
\[\mathcal W_\alpha(\delta_x P_t,\delta_y P_t) \geqslant |\mathbb E_x(X_t) - \mathbb E_y(X_t)| \geqslant |x-y| - 2t = \mathcal W_\alpha(\delta_x,\delta_y)-2t\,,\]
and thus
\[\sup_{x,y\in\R^d,x\neq y} \frac{\mathcal W_\alpha(\delta_x P_t,\delta_y P_t) }{\mathcal W_\alpha(\delta_x,\delta_y)} \geqslant 1\]
for all $t\geqslant 0$.

\subsubsection*{Wasserstein contraction versus Wasserstein convergence}

In the case where $\mu\propto \exp(-U)$ and $U$ is convex at infinity, $\mu$ is known to satisfy a so-called log-Sobolev inequality, see \cite{BakryGentilLedoux}, which is stronger than the Poincaré inequality and implies (combining the exponential decay of the entropy, the $\mathcal T_2$ Talagrand inequality implied by the log-Sobolev one and the $\mathcal W_2$/entropy regularization of \cite{RocknerWang}, see e.g. the proof of \cite[Theorem 2]{MonmarcheGuillinVFP}) that there exist $C,\lambda>0$ such that for all $t\geqslant0$ and \nv{any} probability law $\nu$ on $\R^d$,
\[\mathcal W_2(\nu P_t,\mu) \leqslant C e^{-\lambda t} \mathcal W_2(\nu,\mu)\,.\]
However this convergence toward\nv{s} equilibrium in $\mathcal W_2$ is weaker than the contraction \eqref{eq:contractW} for $\alpha=2$. Besides, we are interested in cases where the Poincaré inequality does not follow from standard arguments, and thus neither does the log-Sobolev inequality. 

Besides, in \cite[Corollary 1.3]{Wang3}, the log-Sobolev inequality is proven from the exponential decay of weighted gradients along the semi-group, but in contrast to the present work it concerns reversible processes, more precisely the perturbation argument used at the end of \cite[Corollary 1.3]{Wang3} requires an explicit invariant measure. In order to use a perturbation argument when the invariant measure of the semi-group $P_t$ is unknown, we can still say that the measure $\phi\mu$, where $\phi$ is bounded above and below by positive constants, is invariant by the semi-group $P_t^\phi  f= \phi^{-1} P_t(\phi f)$, but it is unclear whether this could be used to adapt the proof of \cite{Wang3} to non-reversible cases with unknown $\mu$.

\subsubsection*{Non-reversible sampling}

In the context of Markov Chain Monte Carlo algorithms, a question is the following: given a known target probability measure $\mu\propto \exp(-U)$, one would like to find, among all the drifts $b$ such that $\mu$ is invariant for $L= b\cdot \na + \Delta$, the one for which convergence toward equilibrium is the fastest (here for simplicity we fix the diffusion matrix to be the identity), see e.g. \cite{Lelievre2012,MonmarcheGuillinOU,HHS1,HHS2}. If the convergence toward equilibrium is quantified in term\nv{s} of the $L^2(\mu)$-norm, then
\[\sup_{f\in L^2(\mu)} \frac{\|P_t f - \mu f\|_{L^2(\mu)}}{\|\nv{ f} - \mu f\|_{L^2(\mu)}} \leqslant e^{-t/C_P(\mu)}\,,\]
see e.g. Theorem~\ref{Thm:main2}, and moreover this is an equality in the reversible case. As a consequence, adding a non-reversible part to the generator $-\na U \cdot\na + \Delta$ can only improve the $L^2$ convergence rate. For instance, \nv{as already mentioned in Section~\ref{sec1},}  in the Gaussian case where $U(x) = x\cdot A x$ for some definite positive symmetric matrix $A$, $C_P(\mu)$ is the minimum of the eigenvalues of $A$ while the optimal rate obtained with non-reversible elliptic diffusions is the mean of the eigenvalues of $A$, see \cite{Lelievre2012}.

Alternatively, the efficiency of the process can be measured in term\nv{s} of $\mathcal W_2$-contraction (and this leads to the same conclusion for Gaussian processes).  Let $\mathcal B$ be the set of drifts $b$ which are $K$-Lipschitz and such that $\mu$ is invariant for $b\cdot \na +\Delta$ (in practice, the Lipschitz constant impacts the stability of the numerical schemes used to discretize \eqref{eq:EDS}, and thus the non-reversible part should not be too large).  Then, instead of seeing Theorem~\ref{Thm:main2} as a way to obtain a Poincaré inequality from a Wasserstein contraction, here we can use this result to obtain a constraint on $M$ and $\lambda$ such that \eqref{eq:contractW} holds in term\nv{s} of $\mu$, uniformly over $\mathcal B$. Another way to see this is the following: for a fixed $t>0$ (corresponding to a fixed computational budget), what is the smallest $\gamma(t)$ one can obtain such that there exist a drift $b\in\mathcal B$ such that
\begin{equation}\label{eq:gama}
\forall \nu,\nu'\in\R^d\,,\qquad \mathcal W_2(\nu P_t,\nu' P_t) \leqslant \gamma(t) \mathcal W_2(\nu,\nu')\,,
\end{equation}
where $P_t$ is the semi-group associated to $b$? More generally, for a given $b$ and for $s\geqslant 0$, let
\[\gamma(s) = \sup_{\nu\neq \nu'} \frac{\mathcal W_2(\nu P_s,\nu' P_s)}{\mathcal W_2(\nu,\nu')}\,.\]
 Since $b$ is $K$-Lipschitz, we immediatly obtain from a synchronous coupling that $\gamma(s)\leqslant e^{Ks}$ and, following the proof of Theorem~\ref{Thm:main2},  \eqref{eq:gama} implies that
 \begin{equation}\label{eq:contrainteCp}
C_P(\mu) \leqslant \int_0^\infty \gamma(s)\dd s \leqslant \frac{e^{Kt}-1}{K}\sum_{k=0}^\infty \gamma(t)^k = \frac{e^{Kt}-1}{K(1-\gamma(t))} \,. 
 \end{equation}
Hence the lower bound on the contraction rate
\[ \gamma(t) \geqslant  1 - \frac{e^{Kt}-1}{C_P(\mu)K}\]
uniformly over $\mathcal B$. Of course this is not very informative for large $t$.

In fact, for sampling algorithm, one is more interested in ergodic averages rather than marginal laws at  given times. The bias of an MCMC estimator for a Lipschitz test function is typically bounded as
\[\left|\mathbb E_\nu\po \frac{1}{t}\int_0^t f(X_s) \dd s\pf - \mu f \right| \leqslant \|\na f\|_\infty \frac1t\int_0^t \mathcal W_2(\nu P_s ,\mu) \dd s  \leqslant \|\na f\|_\infty  \mathcal W_2(\nu,\mu) \frac1t \int_0^t \gamma(s)\,.\]
Trying to minimize the right hand side by a suitable choice of drift in $\mathcal B$, in any cases it is not possible to get better than $C_P(\mu)/t + o(1/t)$  due to the constraint \eqref{eq:contrainteCp}. Having a better contraction rate in large times is thus only useful if the estimator is $1/(t-t_0)\int_{t_0}^t f(X_s)\dd s$ for a suitable warm-up time $t_0$.

Finally,  notice that, in Theorem~\ref{Thm:main}, $M\leqslant \sqrt{2}$, and thus $\lambda\leqslant 2 T/C_P(\mu)$, which means that, for instance, the improvement of the $L^2$ decay rate from $T/C_P(\mu)$ to $\lambda$ obtained by \nv{combining} Theorems~\ref{Thm:main} and \ref{Thm:main2} is small.

\nv{
\subsubsection*{Link with a Feynman-Kac eigenvalue problem}
This section has to be credited to the anonymous referee who made the following remark: in Section~\ref{sec:demo-couplage}, writing $u(x) = T+\alpha \omega(x)$,  the proof works as soon as we find $\lambda>0$ and a positive function $u$ such that
\begin{equation}\label{eq:KreinRutman}
Lu(x) - \alpha k(x) u(x)  \leqslant - \lambda u(x)\qquad \forall x\in\R^d\,,
\end{equation}
which is \eqref{eq:conditionPsi}, but in fact looking for $u,\lambda$ such that equality holds in \eqref{eq:KreinRutman} is an eigenvalue problem for the operator $Lu-\alpha ku$, and the Krein-Rutman theorem (which holds under Assumption~\ref{Assu1}, see e.g. \cite[Corollary 4.2]{champagnat2023general}) states that in fact such an eigenpair $u,\lambda$ always exist, with $u>0$.

 The remaining question is whether $\lambda >0$. Since $f_t(x) = e^{-\lambda t} u(x)$ solves
\[\partial_t f_t = Lf_t - \alpha k u\,,\]
we have for $u$ a Feynman-Kac representation
\[u(x) = e^{\lambda t} \mathbb E_x \po e^{-\alpha \int_0^t k(X_s)\dd s} u(X_t) \pf\]
for any $t\geqslant 0$. Normalizing $u$ to have $\mu(u) = 1$, using that $u$ grows at most polynomially (again thanks to \cite[Corollary 4.2]{champagnat2023general} since under Assumption~\ref{Assu1}, $x\mapsto |x|^2$ is a Lyapunov function for $L$) and that $\mu$ has all its polynomial moment finite, we get by integrating the previous inequality with respect to $\mu$ and using the Hölder inequality that
\[1 \leqslant e^{-\lambda t} \po \mathbb E_{\mu} \po e^{- a \alpha \int_0^t k(X_s)\dd s} \pf\pf^{1/a} \po \mu(u^{b})\pf^{1/b} \]
with $1/b+1/a=1$ for all $a>1$, $t\geqslant 0$. This implies that
\[\lambda \geqslant  - \limsup_{t\geqslant 0} \frac{1}{t}\ln \po \mathbb E_{\mu} \po e^{- a\alpha \int_0^t k(X_s)\dd s}  \pf\pf^{1/a}\,, \]
for all $a>1$, and thus for $a=1$, namely
\[\lambda \geqslant \bar\lambda := - \limsup_{t\geqslant 0} \frac{1}{t}\ln   \mathbb E_{\mu} \po e^{- \alpha \int_0^t k(X_s)\dd s}  \pf\,. \]
 This quantity naturally appears when intertwining the semi-group $P_t$ with the gradient, as in \cite{10.3150/12-BEJ433,CattiauxGuillinFathi,arnaudon2018intertwinings} (however, notice that in our case, we are just interested in the asymptotic exponential rate, and the expectation is with respect to the invariant measure $\mu$ instead of taking the supremum over all initial conditions $x\in\R^d$).

It remains to see  whether $\bar \lambda>0$. Notice that, by Jensen's inequality,
\[\bar \lambda \leqslant - \limsup_{t\geqslant 0} \frac{1}{t}    \mathbb E_{\mu} \po  - \alpha \int_0^t k(X_s)\dd s  \pf  =  a \mu(k)\,, \]
which means that having a positive mean curvature $\mu(k)$ is necessary to proceed with our proof based on the synchronous coupling.
 
Under Assumption~\ref{Assu1}, it is classically seen that $V(x) = e^{\delta |x|^2}$ is a Lyapunov function for $L$ (in the sense that $LV \leqslant - r V$ outside some compact for some $r>0$) as soon as $\delta <c/(2T)$.  According to \cite[Theorem 2.3.]{Djellout}, $\mu$ thus satisfies a $T_1$ Talagrand transport inequality (for the Euclidean distance on $\R^d$) with constant $\theta T$ for some $\theta>0$ independent from $T$ (we refer to \cite{Djellout} for definitions and details). Then, we can use   \cite[Corollary 2.4]{guillin2009transportation}  (although it is written for reversible processes, this assumption is only used in its second part; the first part is a direct corollary of \cite[Theorem 2.2]{guillin2009transportation}  where reversibility is not assumed) to bound, for any $v > 0$,
\[\mathbb P_\mu \po - \frac\alpha{t} \int_0^t k(X_{s/T})\dd s \geqslant -\alpha \mu(k) +  v  \pf \leqslant \exp\po- \frac{v^2 t }{4\theta T \alpha^2 \|k\|_{lip}^2}\pf  \]
(where, for consistency with \cite{guillin2009transportation}, we rescaled the process in time so that the corresponding carré du champ is $\Gamma(f)=|\na f|^2$, corresponding to the standard distance, instead of $T|\na f|^2$). Here we have to assume that $k$ is a Lipschitz function, which is not a problem since our proof in Section~\ref{sec:demo-couplage} works if $k$ is replaced by a lower bound of $k$.

Then, we follow the proof of \cite[Corollary 4.4]{CattiauxGuillinFathi}. Assuming that $\mu(k)>0$, we introduce the event
\[\mathcal A = \left\{ - \frac\alpha{t} \int_0^t k(X_{s})\dd s \geqslant  - \alpha \mu(k) + v   \right\}  = \left\{ - \frac\alpha{Tt} \int_0^{Tt} k(X_{s/T})\dd s \geqslant  - \alpha \mu(k) + v   \right\} \] 
for some $v>0$ to be chosen later on. Since $k$ is lower bounded by $-K$,
\begin{eqnarray*}
\mathbb E_{\mu} \po e^{-   \alpha \int_0^t k(X_{s})\dd s}  \pf  & \leqslant & e^{\alpha Kt} \mathbb P(\mathcal A) + e^{ - \alpha \mu(k) t + v t} \\
& \leqslant &  \exp\po \alpha K t - \frac{v^2  t}{4\theta  \alpha^2 \|k\|_{lip}^2}\pf   + e^{ - \alpha \mu(k) t + v t}\,.
\end{eqnarray*}
Taking e.g. $v = \alpha \mu(k)/2$, we end up with
\[\lambda \geqslant \bar \lambda  >0 \]
provided
\[\alpha K < \frac{ \po \mu(k)\pf^2  }{16\theta    \|k\|_{lip}^2}\,. \]
Finally, as already mentioned, since we only need $k$ to be a lower bound of the curvature, we can take it with a Lipschitz constant arbitrary small. However, this reduces $\mu(k)$, which can become negative. But then by assuming $T$ is large enough we can make $\mu(k)$ positive again. More precisely: first, fix a lower bound $\tilde k$ of the real curvature $k$ defined by \eqref{eq:defk} (with $\inf \tilde k = -K$ and $\tilde k(x) = c$ for $x$ large enough) with a Lipschitz constant $\|\tilde k\|_{lip}$ sufficiently small so that 
\[\alpha K < \frac{ c^2  }{64\theta    \|\tilde k\|_{lip}^2}\,. \]
($\theta$ being related to $\mu$, it is not affected by choice of $\tilde k$). Then, since $\mu(\{x\in \R^d, |x| \geqslant R\}) \rightarrow 1$ for all $R>0$ as $T\rightarrow \infty$, there exists $T_0>0$  such that for $T\geqslant T_0$, $\mu(\tilde k) \geqslant c/2$, and thus $\lambda>0$.

At the conclusion of this sketch of proof, we have thus obtained the  following: under Assumption~\ref{Assu1}, for any $\alpha>0$, there exists $T_0>0$ such that the operator $L - \alpha \tilde k$ has an eigenpair $(u,\lambda)$ with $u>0$ and $\lambda>0$ if $T\geqslant T_0$, where $\tilde k$ is some lower bound of the curvature (which implies in particular that $Lu - \alpha k u \leqslant -\lambda u$). 

Now, in order to conclude with a result similar to Theorem~\ref{Thm:main}, it remains to do a bit of work on $u$, which we will not discuss here. Notice that, with an argument which starts with a non-explicit existence of the positive eigenfunction $u$, it is not necessarily easy to end up with  explicit estimates as in Theorem~\ref{Thm:main} (in particular for the constant $M$). However this approach can be applied in a much more general framework, for instance for non-elliptic hypoelliptic diffusion processes. This question will be the topic of a future work. Moreover, this discussion suggests that the high-temperature regime is in fact a necessary condition for the synchronous coupling to contract distances for sufficiently large times under Assumption~\ref{Assu1} (which is consistent with the remarks after Theorem~\ref{Thm:main2}).
}

\bibliographystyle{plain}
\bibliography{biblio}

\begin{thebibliography}{10}

\bibitem{arnaudon2018intertwinings}
Marc Arnaudon, Michel Bonnefont, and Ald{\'e}ric Joulin.
\newblock Intertwinings and generalized brascamp--lieb inequalities.
\newblock {\em Revista Matem{\'a}tica Iberoamericana}, 34(3):1021--1054, 2018.

\bibitem{BakryCattiauxGuillin}
Dominique Bakry, Patrick Cattiaux, and Arnaud Guillin.
\newblock Rate of convergence for ergodic continuous {M}arkov processes:
  {L}yapunov versus {P}oincar\'e.
\newblock {\em J. Funct. Anal.}, 254(3):727--759, 2008.

\bibitem{BakryGentilLedoux}
Dominique Bakry, Ivan Gentil, and Michel Ledoux.
\newblock {\em Analysis and geometry of {M}arkov diffusion operators}, volume
  348 of {\em Grundlehren der Mathematischen Wissenschaften [Fundamental
  Principles of Mathematical Sciences]}.
\newblock Springer, Cham, 2014.

\bibitem{BaudoinGordinaHerzog}
Fabrice {Baudoin}, Maria {Gordina}, and David~P. {Herzog}.
\newblock {Gamma Calculus Beyond Villani and Explicit Convergence Estimates for
  Langevin Dynamics with Singular Potentials}.
\newblock {\em Archive for Rational Mechanics and Analysis}, 241(2):765--804,
  August 2021.

\bibitem{Bobkov}
Sergey~G. Bobkov.
\newblock {Isoperimetric and Analytic Inequalities for Log-Concave Probability
  Measures}.
\newblock {\em The Annals of Probability}, 27(4):1903 -- 1921, 1999.

\bibitem{bonnefont2022note}
Michel Bonnefont and Ald{\'e}ric Joulin.
\newblock A note on eigenvalues estimates for one-dimensional diffusion
  operators.
\newblock {\em Bernoulli}, 28(1):64--86, 2022.

\bibitem{BonnefontJoulinMa}
Michel Bonnefont, Aldéric Joulin, and Yutao Ma.
\newblock Spectral gap for spherically symmetric log-concave probability
  measures, and beyond.
\newblock {\em Journal of Functional Analysis}, 270(7):2456--2482, 2016.

\bibitem{CattiauxGuillinFathi}
Patrick {Cattiaux}, Max {Fathi}, and Arnaud {Guillin}.
\newblock {Self-improvement of the Bakry-Emery criterion for Poincar{\'e}
  inequalities and Wasserstein contraction using variable curvature bounds}.
\newblock {\em arXiv e-prints}, page arXiv:2002.09221, February 2020.

\bibitem{CattiauxGuillinDeviation}
Patrick Cattiaux and Arnaud Guillin.
\newblock Deviation bounds for additive functionals of markov processes.
\newblock {\em ESAIM: Probability and Statistics}, 12:12--29, 2008.

\bibitem{CattiauxGuillin}
Patrick Cattiaux and Arnaud Guillin.
\newblock {\em Functional inequalities via Lyapunov conditions}, page
  274–287.
\newblock London Mathematical Society Lecture Note Series. Cambridge University
  Press, 2014.

\bibitem{CattiauxGuillinSLC}
Patrick Cattiaux and Arnaud Guillin.
\newblock {\em Semi Log-Concave Markov Diffusions}, pages 231--292.
\newblock Springer International Publishing, Cham, 2014.

\bibitem{CattiauxGuillinPAZ}
Patrick Cattiaux, Arnaud Guillin, and Pierre-Andr\'e Zitt.
\newblock Poincar\'e inequalities and hitting times.
\newblock {\em Ann. Inst. Henri Poincar\'e Probab. Stat.}, 49(1):95--118, 2013.

\bibitem{10.3150/12-BEJ433}
Djalil Chafa{\"i} and Ald{\'e}ric Joulin.
\newblock {Intertwining and commutation relations for birth–death processes}.
\newblock {\em Bernoulli}, 19(5A):1855 -- 1879, 2013.

\bibitem{champagnat2023general}
Nicolas Champagnat and Denis Villemonais.
\newblock General criteria for the study of quasi-stationarity.
\newblock {\em Electronic Journal of Probability}, 28:1--84, 2023.

\bibitem{Clubdes5}
Paul-Eric {Chaudru de Raynal}, Manh~Hong {Duong}, Pierre {Monmarch{\'e}},
  Milica {Toma{\v{s}}evi{\'c}}, and Julian {Tugaut}.
\newblock {Reducing exit-times of diffusions with repulsive interactions}.
\newblock {\em arXiv e-prints}, page arXiv:2110.13230, October 2021.

\bibitem{ChenKLS}
Yuansi Chen.
\newblock An almost constant lower bound of the isoperimetric coefficient in
  the kls conjecture.
\newblock {\em Geom. Funct. Anal.}, 31:34–61, 2021.

\bibitem{Djellout}
H.~Djellout, A.~Guillin, and L.~Wu.
\newblock {Transportation cost-information inequalities and applications to
  random dynamical systems and diffusions}.
\newblock {\em The Annals of Probability}, 32(3B):2702 -- 2732, 2004.

\bibitem{DiscreteSticky}
Alain {Durmus}, Andreas {Eberle}, Aur{\'e}lien {Enfroy}, Arnaud {Guillin}, and
  Pierre {Monmarch{\'e}}.
\newblock {Discrete sticky couplings of functional autoregressive processes}.
\newblock {\em arXiv e-prints}, page arXiv:2104.06771, April 2021.

\bibitem{DynkinVol2}
Evgenij~B. Dynkin.
\newblock {\em Markov Processes Volume II}.
\newblock Springer Verlag, 1965.

\bibitem{Eberle2}
Andreas Eberle.
\newblock Reflection coupling and {W}asserstein contractivity without
  convexity.
\newblock {\em Comptes Rendus Mathematique}, 349(19):1101--1104, 2011.

\bibitem{Eberle1}
Andreas Eberle.
\newblock Reflection couplings and contraction rates for diffusions.
\newblock {\em Probab. Theory Relat. Fields}, 166:851--886, 2016.

\bibitem{EberleGuillinZimmer}
Andreas Eberle, Arnaud Guillin, and Raphael Zimmer.
\newblock {Couplings and quantitative contraction rates for Langevin dynamics}.
\newblock {\em The Annals of Probability}, 47(4):1982 -- 2010, 2019.

\bibitem{EberleMajka}
Andreas Eberle and Mateusz~B. Majka.
\newblock {Quantitative contraction rates for Markov chains on general state
  spaces}.
\newblock {\em Electronic Journal of Probability}, 24(none):1 -- 36, 2019.

\bibitem{EberleZimmer}
Andreas Eberle and Raphael Zimmer.
\newblock {Sticky couplings of multidimensional diffusions with different
  drifts}.
\newblock {\em Annales de l'Institut Henri Poincaré, Probabilités et
  Statistiques}, 55(4):2370 -- 2394, 2019.

\bibitem{EthierKurtz}
Stewart~N. Ethier and Thomas~G. Kurtz.
\newblock {\em Markov processes}.
\newblock Wiley Series in Probability and Mathematical Statistics: Probability
  and Mathematical Statistics. John Wiley \& Sons, Inc., New York, 1986.
\newblock Characterization and convergence.

\bibitem{FournierGuillin}
Nicolas Fournier and Arnaud Guillin.
\newblock On the rate of convergence in {W}asserstein distance of the empirical
  measure.
\newblock {\em Probab. Theory Related Fields}, 162(3-4):707--738, 2015.

\bibitem{guillin2009transportation}
Arnaud Guillin, Christian L{\'e}onard, Liming Wu, and Nian Yao.
\newblock Transportation-information inequalities for markov processes.
\newblock {\em Probability theory and related fields}, 144:669--695, 2009.

\bibitem{GuillinWuZhang}
Arnaud {Guillin}, Wei {Liu}, Liming {Wu}, and Chaoen {Zhang}.
\newblock {Uniform Poincar\'e and logarithmic Sobolev inequalities for mean
  field particles systems}.
\newblock {\em arXiv e-prints}, page arXiv:1909.07051, Sep 2019.

\bibitem{MonmarcheGuillinOU}
Arnaud Guillin and Pierre Monmarch{\'e}.
\newblock Optimal linear drift for an hypoelliptic diffusion.
\newblock {\em Electronic Communication of Probability}, 21, 2016.

\bibitem{MonmarcheGuillinVFP}
Arnaud Guillin and Pierre Monmarch{\'e}.
\newblock Uniform long-time and propagation of chaos estimates for mean field
  kinetic particles in non-convex landscapes.
\newblock {\em J Stat Phys}, 185, 2021.

\bibitem{HMWasserstein}
Martin Hairer and Jonathan~C. Mattingly.
\newblock Spectral gaps in {W}asserstein distances and the 2d stochastic
  {N}avier-{S}tokes equations.
\newblock {\em The Annals of Probability}, 36(6):2050--2091, 2008.

\bibitem{HairerMattingly2008}
Martin Hairer and Jonathan~C. Mattingly.
\newblock Yet another look at {H}arris' ergodic theorem for {M}arkov chains.
\newblock In {\em Seminar on {S}tochastic {A}nalysis, {R}andom {F}ields and
  {A}pplications {VI}}, volume~63 of {\em Progr. Probab.}, pages 109--117.
  Birkh\"auser/Springer Basel AG, Basel, 2011.

\bibitem{HKS}
Richard~A. Holley, Shigeo Kusuoka, and Daniel~W. Stroock.
\newblock Asymptotics of the spectral gap with applications to the theory of
  simulated annealing.
\newblock {\em J. Funct. Anal.}, 83(2):333--347, 1989.

\bibitem{HHS1}
Chii-Ruey Hwang, Shu-Yin Hwang-Ma, and Shuenn-Jyi Sheu.
\newblock Accelerating gaussian diffusions.
\newblock {\em Ann. Appl. Probab.}, 3:897--913, 1993.

\bibitem{HHS2}
Chii-Ruey Hwang, Shu-Yin Hwang-Ma, and Shuenn-Jyi Sheu.
\newblock Accelerating diffusions.
\newblock {\em Ann. Appl. Probab.}, 15:1433--1444, 2005.

\bibitem{Kuwada1}
Kazumasa Kuwada.
\newblock Duality on gradient estimates and {W}asserstein controls.
\newblock {\em J. Funct. Anal.}, 258(11):3758--3774, 2010.

\bibitem{Kuwada2}
Kazumasa Kuwada.
\newblock Gradient estimate for {M}arkov kernels, {W}asserstein control and
  {H}opf-{L}ax formula.
\newblock In {\em Potential theory and its related fields}, RIMS
  K\^{o}ky\^{u}roku Bessatsu, B43, pages 61--80. Res. Inst. Math. Sci. (RIMS),
  Kyoto, 2013.

\bibitem{LeeVempala}
Yin~Tat Lee and Santosh~S. Vempala.
\newblock {T}he {K}annan-{L}ov{\'a}sz-{S}imonovits conjecture.
\newblock {\em Current Developments in Mathematics}, 2017.

\bibitem{Lelievre2012}
Tony {Leli{\`e}vre}, Francis {Nier}, and Grigorios~A. {Pavliotis}.
\newblock Optimal non-reversible linear drift for the convergence to
  equilibrium of a diffusion.
\newblock {\em J. Stat. Phys.}, 152(2):237--274, 2013.

\bibitem{JWang2}
Dejun Luo and Jian Wang.
\newblock Exponential convergence in -wasserstein distance for diffusion
  processes without uniformly dissipative drift.
\newblock {\em Mathematische Nachrichten}, 289(14-15):1909--1926, 2016.

\bibitem{Meleard}
Sylvie M{\'e}l{\'e}ard.
\newblock Asymptotic behaviour of some interacting particle systems;
  {M}c{K}ean-{V}lasov and {B}oltzmann models.
\newblock In {\em Probabilistic models for nonlinear partial differential
  equations}, pages 42--95. Springer, 1996.

\bibitem{MenzSchlichting}
Georg Menz and Andr{\'e} Schlichting.
\newblock Poincar\'e and logarithmic {S}obolev inequalities by decomposition of
  the energy landscape.
\newblock {\em Ann. Probab.}, 42(5):1809--1884, 2014.

\bibitem{MeynTweedieLivre}
Sean Meyn and Robert~L. Tweedie.
\newblock {\em Markov chains and stochastic stability}.
\newblock Cambridge University Press, Cambridge, second edition, 2009.
\newblock With a prologue by Peter W. Glynn.

\bibitem{MonmarcheElementary}
Pierre {Monmarch{\'e}}.
\newblock {Elementary coupling approach for non-linear perturbation of Markov
  processes with mean-field jump mechanims and related problems}.
\newblock {\em arXiv e-prints}, page arXiv:1809.10953, September 2018.

\bibitem{MonmarcheContraction}
Pierre {Monmarch{\'e}}.
\newblock {Almost sure contraction for diffusions on $\mathbb R^d$. Application
  to generalised Langevin diffusions}.
\newblock {\em arXiv e-prints}, page arXiv:2009.10828, September 2020.

\bibitem{RocknerWang}
Michael R{\"o}ckner and Feng-Yu Wang.
\newblock Log-harnack inequality for stochastic differential equations in
  hilbert spaces and its consequences.
\newblock {\em Infinite Dimensional Analysis, Quantum Probability and Related
  Topics}, 13(01):27--37, 2010.

\bibitem{Sznitman}
Alain-Sol Sznitman.
\newblock Topics in propagation of chaos.
\newblock In Paul-Louis Hennequin, editor, {\em Ecole d'Et{\'e} de
  Probabilit{\'e}s de Saint-Flour XIX --- 1989}, pages 165--251, Berlin,
  Heidelberg, 1991. Springer Berlin Heidelberg.

\bibitem{Tugaut2014}
Julian Tugaut.
\newblock Phase transitions of {M}c{K}ean-{V}lasov processes in double-wells
  landscape.
\newblock {\em Stochastics}, 86(2):257--284, 2014.

\bibitem{VonRenesseSturm}
Max-K. von Renesse and Karl-Theodor Sturm.
\newblock Transport inequalities, gradient estimates, entropy and ricci
  curvature.
\newblock {\em Communications on Pure and Applied Mathematics}, 58(7):923--940,
  2005.

\bibitem{Wang2}
Feng-Yu Wang.
\newblock Logarithmic sobolev inequalities on noncompact riemannian manifolds.
\newblock {\em Probability Theory and Related Fields}, 109:417--424, 1997.

\bibitem{Wang3}
Feng-Yu Wang.
\newblock Modified curvatures on manifolds with boundary and applications.
\newblock {\em Potential Analysis}, 41:699--714, 2011.

\bibitem{Wang}
Feng-Yu Wang.
\newblock Coupling and applications.
\newblock {\em Stochastic Analysis and Applications to Finance}, pages
  411--424, 2012.

\bibitem{JWang1}
Jian Wang.
\newblock {$L^{p}$-Wasserstein distance for stochastic differential equations
  driven by Lévy processes}.
\newblock {\em Bernoulli}, 22(3):1598 -- 1616, 2016.

\bibitem{Yosida}
Kosaku Yosida.
\newblock {\em Functional analysis}.
\newblock Classics in Mathematics. Springer-Verlag, Berlin, 1995.
\newblock Reprint of the sixth (1980) edition.

\end{thebibliography}
\end{document}